\documentclass[11pt,a4paper,reqno]{amsart}
 \usepackage{amsfonts,amsthm, amscd, epsfig, amsmath, amssymb,enumerate}
  \usepackage[dvips]{psfrag}
 \numberwithin{equation}{section}

\author{Alessandra  Faggionato}
\address{Alessandra Faggionato. Dipartimento di Matematica ``G. Castelnuovo", Universit\`a   ``La
  Sapienza''. P.le Aldo Moro  2, 00185  Roma, Italy. e--mail:
  faggiona@mat.uniroma1.it}

\newtheorem{theo}{Theorem}

\newtheorem{prop}{Proposition}
\newtheorem{lemma}{Lemma}
\newtheorem{coro}{Corollary}
\newtheorem{rem}{Remark}

\newcommand{\CC}{{\mathbb C}}\newcommand{\NN}{{\mathbb N}}
\newcommand{\RR}{{\mathbb R}}
\newcommand{\ZZ}{{\mathbb Z}}

\newcommand{\Aa}{{\mathcal A}}
\newcommand{\Pp}{{\mathcal P}}
\newcommand{\Qq}{{\mathcal Q}}
\newcommand{\PP}{{\mathbf P}}

\newcommand{\EE}{{\mathbf E}}

\newcommand{\Ff}{{\mathcal F}}
\newcommand{\Ww}{{\mathcal W}}
\newcommand{\Ss}{{\mathcal S}}

\newcommand{\Nn}{{\mathcal N}}
\newcommand{\Cc}{{\mathcal C}}

\newcommand{\Ll}{{\mathcal L}}

\newcommand{\II}{{\mathbb I}}
\let\a=\alpha \let\b=\beta   \let\d=\delta  \let\e=\varepsilon
 \let\g=\gamma     
\let\l=\lambda
   
\let\o=\omega
\let\p=\pi  

\let\s=\sigma  
\let\t=\tau
  
\let\D=\Delta   \let\G=\Gamma
\let\L=\Lambda

\begin{document}
\title[$h$--slopes  of drifted Brownian
motion]{The alternating marked point process of $h$--slopes  of the
drifted Brownian motion}
\begin{abstract}
We show that the slopes between  $h$--extrema of the drifted 1D
Brownian motion  form a stationary alternating marked point process,
extending the result of J. Neveu and J. Pitman for the non drifted
case. Our analysis covers the results on the statistics of
$h$--extrema obtained by P. Le Doussal, C. Monthus and D. Fisher via
a Renormalization Group analysis and  gives a complete description
of the slope between $h$--extrema covering the origin by means of
the Palm--Khinchin theory. Moreover, we analyze the behavior of the
Brownian motion near its  $h$--extrema.

\bigskip



\end{abstract}

\maketitle

\date{today}

\section{Introduction}
Let $B$ be a two--sided standard  Brownian motion with drift $-\mu$.
  Given $h>0$ we  say that $B$
admits an {\em $h$--minimum} at $x\in \RR$, and that $x$ is a point
of $h$--minimum,  if there exist $u<x<v$ such that $ B_t \geq B_x $
for all $t \in [u,v]$, $B_u \geq B_x +h$ and $B_v\geq B_x+h$.
Similarly, we say that $B$ admits an {\em $h$--maximum} at $x\in
\RR$, and that $x$ is a point of $h$--maximum,  if there exist
$u<x<v$ such that $ B_t \leq  B_x $ for all $t \in [u,v]$, $B_u \leq
B_x -h$ and $B_v\leq B_x-h$. We  say that $B$ admits an {\em
$h$--extremum} at $x\in \RR$, and that $x$ is a point of
$h$--extremum, if $x$ is a point of $h$--minimum or a point of
$h$--maximum. Finally,  the truncated  trajectory $B$ going from an
$h$--minimum to an $h$--maximum will be called upward $h$--slope,
while the truncated  trajectory $B$ going from an $h$--maximum to an
$h$--minimum will be called downward $h$--slope.

\smallskip

Our first object of investigation is the statistics  of $h$--slopes.
 The non drifted case $\mu=0$  has been studied in
\cite{NP}. Here we assume  $\mu\not =0$ and  show  (see Theorem
\ref{napoleone})  that the statistics of $h$--slopes is well
described by a stationary alternating    marked simple point process
on $\RR$ whose points are the points of $h$--extrema of the Brownian
motion, and each point $x$ is marked by the $h$--slope going from
$x$ to the subsequent point of $h$--extremum. We will show that the
$h$--slopes are independent and specify the laws $P^\mu_+$,
$P^\mu_-$ of upward $h$--slopes and  downward  $h$--slopes non
covering the origin, respectively. The $h$--slope covering the
origin shows a different distribution  that can be derived by means
of the  Palm--Khinchin theory \cite{DVJ}, \cite{FKAS}.

\smallskip

 Our proof is based both on fluctuation theory for L\'{e}vy processes,
 and on the
theory of marked simple point processes. The part of fluctuation
theory follows strictly the scheme of \cite{NP} and can be
generalized to spectrally one--sided  L\'{e}vy processes, i.e. real
valued random processes with stationary independent increments and
with  no positive jumps or with no negative jumps \cite{B}[Chapter
VII]. In fact, some of the identities of Lemma \ref{davide} and
Proposition \ref{calcoli} below have already been obtained with more
sophisticated methods for general spectrally one--sided  L\'{e}vy
processes (see \cite{Pi}, \cite{AKP}, \cite{C} and references
therein). On the other hand, the description of the  $h$--slopes as
a stationary alternating marked simple point process allows to use
the very powerful Palm--Khinchin theory, which extends  renewal
theory and leads to a complete description of the $h$--slope
covering the origin. This analysis can be easily extended to more
general L\'{e}vy processes, as the ones treated in \cite{C}.

\smallskip

As discussed in Section \ref{rengroup}, our results concerning the
statistics of $h$--extrema of drifted Brownian motion correspond to
the ones obtained in \cite{DFM} via a non rigorous Real Space
Renormalization Group method applied to Sinai random walk with a
vanishing bias. In addition of a rigorous derivation, we are able
here  to describe also the statistics of the $h$--slopes, lacking in
\cite{DFM}.

\smallskip

In section \ref{autunno} (see Theorem \ref{bonaparte}), we analyze
the behavior of the drifted Brownian motion  around  its
$h$--extrema. While in the non--drifted case a generic  $h$--slope
non covering the origin behaves  in proximity of its  extremes  as a
3--dimensional Bessel process, in the drifted case it behaves  as a
process with a cothangent drift, satisfying the SDE
\begin{equation}\label{silurino}
\begin{cases}
& d X_t = d\b_t \pm \mu \coth \left( \mu X _t \right) dt\,,\qquad
t\geq 0\,,\\
& X_0 =0 \,,
\end{cases}
\end{equation}
 where $\b_t$ is an independent standard Brownian motion
and the sign in the r.h.s. depends on the kind of $h$--slope
(downward or upward) and on the kind of $h$--extrema ($h$--minimum
or $h$--maximum). In addition, we show that the process
(\ref{silurino}) is simply the Brownian motion on $[0,\infty)$,
starting at the origin, with drift $\pm \mu$, Doob--conditioned to
hit $+\infty$ before $0$.

\smallskip

The interest in the statistics of $h$--slopes and their behavior
near to the extremes comes also from the fact that, considering the
diffusion in a drifted Brownian potential,  the piecewise linear
path obtained by connecting the $h$--extrema of the Brownian
potential is the effective potential for the diffusion at large
times \cite{BCGD}.



\section{Statistics of $h$--slopes of  drifted Brownian
motion}\label{statistica1000}

 Given $\mu, x  \in
\RR$ we denote by $\PP ^\mu _x$  the law on $C(\RR , \RR)$ of the
standard  two--sided   Brownian motion $B$  with drift $-\mu$ having
value $x$ at time zero, i.e. $B_t=x+B_t^*-\mu t $ where $B^*:
\RR\rightarrow \RR$ denotes the  two--sided Brownian motion s.t.
$B_t$ has  expectation zero and variance $t$.
 We denote the  expectation w.r.t. $\PP^\mu_x$  by
$\EE_x^\mu$. If $\mu = 0$ we simply write $\PP_x$, $\EE_x $.

\smallskip

Recall the definitions of $h$--maximum, $h$--minimum and
$h$--extremum given in the Introduction. It is simple to verify that
$\PP_x^\mu$--a.s. the set of points of  $h$--extrema  is locally
finite, unbounded from below and from above, and that points of
$h$--minima alternate with points of $h$--maxima.  The $h$--slope
between two consecutive points of $h$--extrema $\a$ and $\b$ is
defined as the truncated trajectory $\g:=\bigl( B_t\,:\, t \in
[\a,\b]\bigr)$. We call  it an {\sl upward} slope if $\a$ is a point
of $h$--minimum (and consequently $\b$ is a point of $h$--maximum),
otherwise we call it {\sl downward} slope. The length $\ell (\g)$
and the height $h(\g)$ of the slope $\g$ are  defined as $\ell (\g)=
\b-\a$ and $h(\g)=|\g(\b)-\g(\a)|$, respectively. Moreover, to the
slope $\g$ we associate the translated path
 $\theta (\g):= \bigl( B_{t+\a}-B_\a \,:\, t \in [0,\b-\a]\bigr)$.
 With some abuse of notation (as in the Introduction) we call also
 $\theta (\g)$ the $h$--slope between the points of $h$--extrema
 $\a$ and $\b$. When the context can cause some ambiguity, we will
 explicitly distinguish between the $h$--slope $\g$ and the
 translated path $\theta (\g)$.

 \smallskip

Finally, we introduce the following notation:
 given
$\a\in \RR$, the constant $\hat \a$ is defined
 as
\begin{equation}\label{cappuccio}
\hat \a = \a +\mu^2/2\,.
\end{equation}

\subsection{The building blocks of the $h$--slopes}.
 Given a two--sided Brownian $B$ with law $\PP^\mu_0$ we define the
 random variables $b_t, \t,\b,\s$ as follows (see  figure
 \ref{coniglio2}):
\begin{equation}\label{esameinf}
\begin{cases}
b_t = \min \left \{ B_s \,:\, 0\leq s\leq t \right\}\,,\\
\t = \min \left\{ t\geq 0\,:\, B_t=b_t +h \right\}\, ,\\
\b=b_\t=\min \left\{ B_s\,:\, 0\leq s\leq \t\right\}\, ,\\
\s= \max\left\{s\,:\, s\leq \t,\; B_s  =\b\right\} \, .
\end{cases}
\end{equation}
Note that $\PP^\mu _0$--a.s. there  exists a unique time $s\in
[0,\t]$ with $B_s=\b$, which by definition coincides with  $\s$.

\begin{figure}[!ht]
    \begin{center}
       \psfrag{a}[l][l]{$\beta$}
       \psfrag{b}[l][l]{$\sigma$}
       \psfrag{c}[l][l]{$\tau$}
      \psfrag{d}[l][l]{$h$}
       \includegraphics[width=7cm]{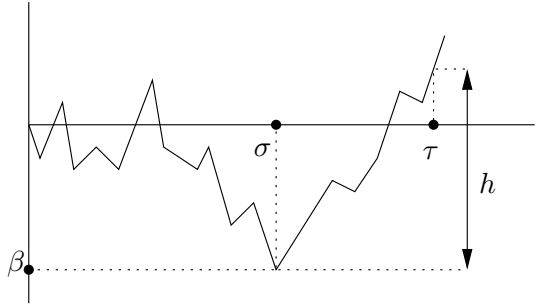}
      \caption{The random variables $\beta $, $\sigma  $, $\tau $.}
    \label{coniglio2}
    \end{center}
  \end{figure}

Our analysis of the statistics of $h$--slopes for the drifted
Brownian motion is based on the following  lemma which extends to
the drifted case  the lemma in Section 1 of \cite{NP}:
\begin{lemma}\label{davide}
Let $\mu \not = 0$, $\hat\a>0$ and $x>0$. Under   $\PP^\mu _0$, the
two trajectories
$$\left(B_t\, ,\,0\leq t\leq \s \right),\qquad
\left(B_{\s+t}-\b\,,\, 0\leq t \leq \t-\s \right)$$ are independent;
in particular  $(\b,\s)$ and $\t-\s$ are independent.

Furthermore $-\b $ is exponentially distributed with mean
\begin{equation}\label{lapla}
\EE_0^\mu (-\b)=\frac{ \sinh (\mu h )}{ \mu e^{-\mu  h }},
\end{equation}
and
\begin{equation}
  \EE_0^\mu\left[ \exp (-\a \s )\,\large{|}\, \b =-x \right]=
\exp \left\{ -x\left[ \sqrt{2\hat\a} \coth \left(\sqrt{2\hat \a}
h\right) -\mu \coth(\mu h )\right] \right\}
  .\label{lapla1}
  \end{equation}
  In particular,  $    \EE_0^\mu\left( \exp (-\a \s )\right)$ is finite if and only if   \begin{equation}\label{mom_finito}
    \sqrt{2\hat \a} \coth \left( \sqrt{2\hat\a}h \right) >\mu  \,.
  \end{equation}
If (\ref{mom_finito}) is fulfilled, then
  \begin{equation}\label{lapla2}
    \EE_0^\mu\left( \exp (-\a \s )\right)=
     \frac{ \mu e^{-\mu h }}{
  \sinh (\mu h )  \left(
\sqrt{2\hat \a} \coth \left( \sqrt{2\hat\a}h \right) -\mu \right) }.
   \end{equation}
  Finally,  it holds
    \begin{equation}
    \EE_0^\mu \left( \exp \left( -\a(\t-\s)\right) \right) =\frac{\sqrt{2\hat \a}}{\mu}
\frac{  \sinh (\mu h ) }{ \sinh\left(\sqrt{2\hat \a} h \right) }
.\label{lapla3}
\end{equation}
\end{lemma}
 The proof of the above lemma is based on excursion theory and
indentities  concerning hitting times of the drifted Brownian
motion. It  will be given in Section \ref{pierpa}.



\medskip

\subsection{The probability  measures $P^{\mu}_+$ and $P^{\mu}_- $ on $h$--slopes}


We define  the path space   $\Ww$  as the set
$$
\Ww= \cup _{T\geq 0} C( [0,T])\,.
$$
Given $\g \in \Ww$, we define $\ell (\g)$  as the nonnegative number
such that $\g \in C[0, \ell (\g) ]$ and we define the path $\g^*:
[0,\infty)\rightarrow \RR $ as
$$
\g^*_t =\begin{cases} \g_t & \text{ if } 0\leq t \leq \ell (\g)\,,\\
\g_{\ell(\g)}   & \text{ if } t \geq \ell (\g)\,. \end{cases}
$$
Then the space $\Ww$ is a Polish space endowed of the metric $d_\Ww$
defined as
$$
d_{\Ww}(\g_1, \g_2)= |\ell(\g_1)-\ell (\g_2)|+ \|\g^*_1- \g^*_2
\|_\infty \,.
$$

 On $\Ww$ we define the Borel probability measures $P^\mu_+$,
$P^\mu_- $ as follows. Let $B,B'$ be independent Brownian motions
with law $\PP^\mu_0$. Recall the definition (\ref{esameinf}) of $\t,
\b, \s$ and define $b_t', \t',\b', \s'$ as (see figure
\ref{pezzobis}):
\begin{equation}\label{esameinfbis}
\begin{cases}
b'_t=\max\left\{ B_s'\,:\, 0\leq s \leq t \right\}\,,\\
\t' = \min \left\{ t\geq 0\,:\, B'_t=b'_t -h \right\}\, ,\\
\b'=b'_\t=\max \left\{ B'_s\,:\, 0\leq s\leq \t'\right\}\, ,\\
\s'= \max\left\{s\,:\, s\leq \t',\; B'_s  =\b'\right\} \, .
\end{cases}
\end{equation}
Then $P^\mu_+$
 is the law of the path $\g$, with $\ell(\g)= \t-\s+\s'$, defined as
\begin{equation}
\g_t\,=\,
\begin{cases}
B_{\s+t}-\b \,, & \text{ if }  t \in [0, \t-\s ]\, ,\\
B'_{t-(\t-\s)}+h \,, & \text{ if } t \in [\t-\s, \t-\s +\s']\,,
\end{cases}
\end{equation}
while  $P^\mu_- $ is the law of the path $\g$,  with $\ell (\g)=\t'
-\s' +\s $, defined as
\begin{equation}
\g_t\,=\,
\begin{cases}
B'_{\s'+t}-\b '\,, & \text{ if } t \in [0, \t'-\s']\,,\\
B_{t-(\t'-\s') }- h \,, & \text{ if } t \in [\t'-\s', \t'-\s'+\s]\,.
\end{cases}
\end{equation}
  Note that $P^\mu_- $  equals the law
of the path $-\g $ if $ \g$ is chosen with law $P^{-\mu}_+ $.

\begin{figure}[!ht]
    \begin{center}
       \psfrag{a}[l][l]{$\beta'$}
       \psfrag{b}[l][l]{$\sigma'$}
       \psfrag{c}[l][l]{$\tau'$}
     \psfrag{d}[l][l]{$h$}
       \includegraphics[width=7cm]{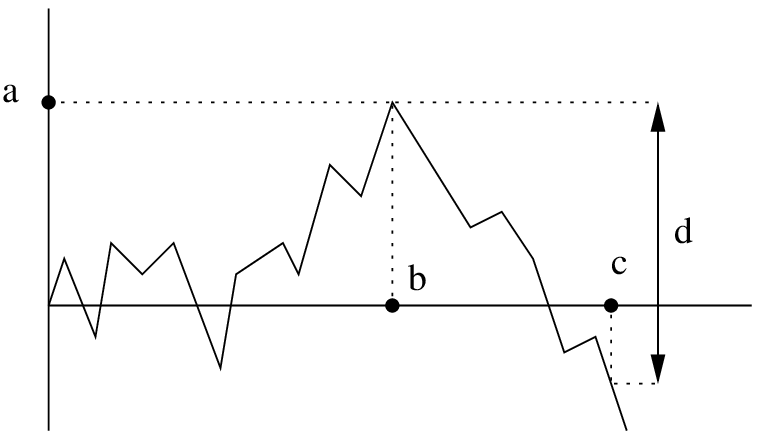}
      \caption{The random variables $\beta ' $, $\sigma ' $, $\tau '$.}
    \label{pezzobis}
    \end{center}
  \end{figure}

We  introduce two  disjoint subsets $\Ww_+$ and $\Ww_-$ of $\Ww$:
\begin{align*}
& \Ww_+ = \left\{ \g \in\Ww\,:\, \g_0 =\min \{\g_t\,:\, t \in
[0,\ell(\g)]\}=0 \,, \;\; \g_{\ell(\g)} = \max \{\g_t\,:\, t \in
[0,\ell(\g)]\}\geq h
\right\}\,, \\
& \Ww_- = \left\{ \g \in\Ww\,:\, \g_0 =\max \{\g_t\,:\, t \in
[0,\ell(\g)]\}=0 \,, \;\; \g_{\ell(\g)} = \min  \{\g_t\,:\, t \in
[0,\ell(\g)]\}\leq -  h  \right\}\,.
\end{align*}
Then the probability measure $P^{\mu}_\pm$ is concentrated on
$\Ww_\pm$. Below we will prove that, given the two--sided BM with
law $\PP^\mu_0$,   $P^{\mu}_+$ is the law of the generic upward
$h$--slope not covering the origin, while  $P^{\mu}_-$ is the law of
the generic downward $h$--slope not covering the origin.

\smallskip

We collect in what follows some results derived from Lemma
\ref{davide} and straightforward computations, which will be useful
in what follows:


\begin{prop}\label{calcoli} Fix $\mu\not =0$.
Let $(\ell_+, \zeta_+ )$ be the  random vector distributed as $
\bigl(\ell (\g), \g _{\ell(\g)}-h \bigr)$ where $\g$ is chosen with
law $P^{\mu}_+$ and let $(\ell_-, \zeta_-)$ be the random vector
distributed as $\bigl( \ell(\g), -(\g_{\ell(\g)}+h)  \bigr)$ where
$\g$  is chosen with law $P^{\mu}_-$.

\smallskip

Then $\zeta_+$, $\zeta_-$ are  exponential variables of mean
\begin{equation}\label{medie}
 \EE  (\zeta_+) = \frac{ \sinh (\mu h )}{
\mu e^{\mu  h }}, \qquad  \EE (\zeta_-) = \frac{ \sinh (\mu h )}{
\mu e^{-\mu  h }}\,.
\end{equation}
 Fix $a$ such that  $
\hat \a:=\a+\mu^2/2 >0$. Then  for all $x>0$
\begin{equation}\label{uffa}
\EE  \left( e^{-\a \ell_{\pm}} \,\Large{|}\,\zeta_\pm = x
\right)=\frac{ \sqrt{2\hat \a}}{\mu } \frac{\sinh(\mu h ) }{\sinh
\left( \sqrt{2\hat \a } h \right)} \exp \left\{ -x
\left(\sqrt{2\hat\a} \coth \left(\sqrt{2\hat \a} h\right) -\mu
\coth(\mu h ) \right) \right\}.
\end{equation}
In particular,  the expectation  $ \EE \left( e^{-\a \ell_{\pm}-\l
\zeta_{\pm } }\right)$ is finite if and only if
\begin{equation}\label{mom_condbis}
 \sqrt{2\hat \a}\coth\left(\sqrt{2\hat \a}h\right) +(\l\pm \mu) >0.
 \end{equation}
If (\ref{mom_condbis}) is fulfilled, then
\begin{equation}\label{mazinga}
 \EE  \left( e^{-\a \ell_{\pm}-\l  \zeta_{\pm } }\right)= \frac{
\sqrt{2\hat \a } e^{\pm  \mu h  }}{
 \sqrt{2\hat \a}\cosh\left(\sqrt{2\hat \a}h\right) +(\l\pm \mu) \sinh  \left(\sqrt{2\hat \a}h\right)
}.
\end{equation}
Hence
\begin{align}
& \EE\left( \ell_+  \right) =\mu ^{-2}
\left(\mu h -\sinh(\mu h ) e^{-\mu h }  \right),\label{zeta1}\\
&  \EE\left( \ell_-\right) =\mu^{-2}  \left(e^{\mu h }\sinh(\mu h
)-\mu h  \right)\,.\label{zeta2}
\end{align}

\medskip

Consider  $\ell:=\ell_-+\ell_+$, where $\ell_-$ and $\ell_+$  are
chosen independently. Then,
\begin{equation}\label{mangobis}
\EE(\ell)= \frac{2}{\mu^2 }\sinh ^2 (\mu h )\,.
\end{equation}
Given $\a\in \RR$, $ \EE \bigl(e^{-\a \ell}\bigr ) $ is finite if
and only if
\begin{equation}\label{mom_condtris}
\begin{cases}
& \hat \a: =\a+\mu^2 /2>0\, ,  \\
& 2\a \cosh^2 (\sqrt{2\hat\a} h )+\mu^2 >0\, .
\end{cases}
\end{equation}
If (\ref{mom_condtris}) is fulfilled, then
\begin{equation}\label{mango}
\EE\bigl(e^{-\a \ell}\bigr ) = \frac{2\hat\a}{ 2\a \cosh^2
(\sqrt{2\hat\a} h )+\mu^2 }.
\end{equation}
\end{prop}

\begin{rem}
Due to the identity
$$ 2\a \cosh^2 (\sqrt{2\hat\a} h )+\mu^2=\frac{ \cosh ^2
 (\sqrt{2\hat \a}h )}{h^2}
\left( 2\hat\a h^2 -\mu^2  h^2 \tanh^2 (\sqrt{2\hat\a} h )
\right)\,,
$$
by straightforward computations one can check that for $\mu h >1$
condition
 (\ref{mom_condtris}) is fulfilled if and only if
$\a > -\mu^2 /2 + y_*^2 / (2h^2)$, where  $y_*$ is the only positive
solution of the equation $ y=\mu h \tanh (y)$. If $\mu h \leq 1$
then condition  (\ref{mom_condtris}) is fulfilled if and only if
$\a> -\mu^2 /2 $.
\end{rem}

\bigskip

 \subsection{The stationary alternating marked simple point process
 $\Pp^\mu$\,.}\label{pediatra}

 We denote by $\Nn$ the space of sequences
$\xi=\bigl\{( x_i, \g_i) \, :\, i \in \ZZ\bigr\}$  such that 1)
$(x_i , \g_i ) \in \RR\times \Ww$, 2)   $x_i<x_{i+1}$, 3)  $x_0 \leq
0< x_1$ and 4) $\lim _{i\rightarrow \pm \infty } x_i =\pm \infty$.
In what follows,  $\xi$ will be often identified  with the counting
measure $ \sum _{i\in \ZZ} \d_{(x_i, \g_i)}$ on $\RR \times \Ww$.

 $\Nn$ is a measurable
space with  $\s$-algebra of measurable sets  generated by
$$
\left\{ \xi \in \Nn \,:\, \xi (A \times B) =j \right\}\,,\qquad
A\subset \RR \text{ Borel},\; B \subset \Ww \text{ Borel},\; j \in
\NN \,.
$$
One can characterize the  above $\s$--algebra  as follows. Consider
the space $\Ss:=(0,\infty) \times (0,\infty)^{\ZZ } \times \Ww
^{\ZZ}$ as a measurable space with $\s$--algebra of measurable sets
given by the  Borel subsets associated to the product topology. Call
$\Ss'$ the subset of $\Ss$ given by the elements where the first
entry is not larger than the entry with index $0$ of the factor
space $(0,\infty)^\ZZ$.  Then by the same arguments leading to
\cite{DVJ}[Proposition 7.1.X] one can prove that the map
\begin{equation}\label{pmassa}
\Nn  \ni \{(x_i , \g_i )\} _{i \in \ZZ} \rightarrow \{x_1\} \times
\{\t_i \}_{i\in \ZZ} \times \{
 \g_i \} _{i \in \ZZ} \in \Ss ' \,,
\qquad \t_i:= x_{i+1}-x_i\,,
\end{equation} is bijective and both ways measurable. We note that
the introduction of $\Ss'$ is due to the constrain $x_1\leq \t_0$.

\smallskip

Let us define $\Pp ^\mu_{0,\pm} $ as the law of the sequences
$\{(x_i, \g_i)\}_{i\in \ZZ}  \in \Nn $  such that
\begin{itemize}
\item  $\{\g_i\}_{i\in \ZZ}$ are independent  random paths,
\item $\{\g_{2i} \}_{i \in \ZZ}$
are i.i.d. random paths with law $P^\mu_\pm $,
\item $\{\g_{2i+1}\}_{i \in \ZZ}$ are i.i.d. random paths with law $P^\mu_\mp$,
\item   $x_0=0$,
\item
$x_{i+1}-x_i = \ell (\g_i)$.
\end{itemize}
Note that $\Pp^\mu_{0,\pm} $ is concentrated on the measurable
subset $\Nn_0$ defined as
$$ \Nn _0 \,=\, \left\{\, \{ (x_i,\g_i)\}_{i\in \ZZ}\,:\, x_0 =0
\right\}\,.
$$
Finally we consider the convex combination
$$
\Pp ^\mu _0 = \frac{1}{2} \Pp ^\mu _{0,+}+ \frac{1}{2} \Pp^\mu
_{0,-}\,.
$$
Let $\theta:\Nn  \rightarrow \Nn $ and, for all $t \in \RR$, let
$T_t : \Nn \rightarrow \Nn $ be the maps defined as
$$
\theta \xi = \sum _{i \in \ZZ} \d _{(x_i-x_1, \g_i)}\,,\qquad  T_t
\xi = \sum _{i \in \ZZ} \d _{(x_i+t , \g_i)}\, \qquad \text{ if }
\xi = \sum_{i \in \ZZ} \d _{(x_i, \g_i)}\,.
$$
We stress that the above  translation map $T_t$   coincides with the
map $T_{-t}$ of  \cite{FKAS} and with the map $S_{-t}$ of
\cite{DVJ}.
 A probability measure $\Qq$ on $\Nn$ is called stationary if
$T_t \Qq (A):= \Qq (T_tA)=\Qq(A)$ for all $t\in \RR$ and all
$A\subset \Nn$ measurable, while it is called $\theta$--invariant if
$\theta \Qq( A):= \Qq ( \theta A)= \Qq (A)$ for all  $A\subset \Nn $
measurable.

Note that $\Pp^\mu _0$ is $\theta$--invariant. Moreover, due to
(\ref{mangobis}) in  Lemma \ref{davide},
\begin{equation}
\EE_{\Pp^\mu _0} ( x_1) = \EE _{\Pp^\mu  _0} (\ell (\g_0))= \EE
(\ell)/2=  \sinh ^2 (\mu h ) /\mu^2\,.
\end{equation}
Hence, due to the  Palm--Khinchin theory (see  Theorem 1.3.1 and
formula (1.2.15) in \cite{FKAS}, and  Theorem 12.3.II in \cite{DVJ})
there exists a unique stationary measure $\Pp^\mu$ on $\Nn $ such
that
\begin{equation}\label{spaghetti}
\begin{split}
\Pp ^\mu (A) & = \frac{ 2}{\EE (\ell)   } \EE_{\Pp^\mu _0} \left[
\int _0 ^{x_1} \chi \left( T_{-t} \left( \{x_i,\g_i\}_{i\in
\ZZ}\right) \in A\, \right) dt \right]\,,\\
& =\frac{  \mu^2}{  \sinh ^2 (\mu h ) } \EE_{\Pp^\mu _0} \left[ \int
_0 ^{x_1} \chi \left( T_{-t} \left( \{x_i,\g_i\}_{i\in \ZZ}\right)
\in A\, \right) dt \right]\,, \end{split}
\end{equation}
where $\chi (\cdot )$ denotes the characteristic function. We simply
say that {\sl $\Pp ^\mu $ is the law of the  stationary alternating
  marked simple point process on $\RR$ with alternating mark laws
given by $P^\mu _+, P^\mu_-$}. The probability measure $\Pp^\mu_0$
is the so called {\sl Palm distribution} associated to $\Pp^\mu$.

\medskip


   One can write
\begin{equation}\label{vito1}
\Pp^\mu (\cdot) = \Pp ^\mu ( \cdot \,|\, \g_0 \in \Ww_+) \Pp
^\mu(\g_0 \in \Ww_+)+\Pp^\mu  ( \cdot \,|\, \g_0 \in \Ww_-) \Pp^\mu
(\g_0 \in \Ww_-)\,.
\end{equation}
From (\ref{spaghetti}) we obtain that
\begin{equation}\label{parto1}
 \Pp ^\mu( \g_0 \in \Ww _\pm )=\frac{2\EE_{\Pp^\mu _0} \left[ x_1 \chi ( \g_0 \in \Ww _\pm) \right]}{ \EE (\ell) } =
\frac{\EE_{P^\mu _\pm} \bigl( \ell (\g) \bigr)}{ \EE (\ell) }=
\frac{ \EE(\ell_\pm )}{\EE (\ell) } \,.
\end{equation}
Hence,  by (\ref{zeta1}) and (\ref{zeta2}) of Lemma \ref{davide}, we
can conclude that
\begin{equation}
 \Pp ^\mu( \g_0 \in \Ww _\pm   )\,=\,
  \frac{\pm \mu h \mp  \sinh (\mu h ) e^{\mp \mu h } }{2
\sinh ^2 (\mu h ) }\,.
\end{equation}
In order to   describe the conditional probability measure
 $\Pp^\mu ( \cdot \,|\, \g_0 \in \Ww_\pm)$ we first observe that
$x_1$ and $\{\g_i \}_{i\in \ZZ} $ univocally determined the set $
\{(x_i, \g_i )\}_{i \in \ZZ}$. Moreover from
 (\ref{spaghetti}) we derive  that, given Borel subsets $A\subset
 \RR$, $B_j\subset \Ww$ for $-m\leq j\leq n $, it holds
 \begin{multline*}
\Pp^\mu ( x_1\in A, \g_j \in B_j \; \forall j: -m\leq j\leq n \,|\,
\g_0 \in \Ww_\pm)=\\\EE_{\Pp ^\mu_{0,\pm}} \left(\int _0 ^{x_1}
\chi(x_1-t\in A)dt ,  \g_j \in B_j  \; \forall j: -m\leq j\leq
n\right)/ \EE (\ell_\pm) =\\
\EE _{P_\pm ^\mu} \left(\int _0 ^{\ell (\g) } \chi( t\in
A)dt , \g\in B_0 \right) \bigl[ \prod _{\substack{-m\leq j \leq n\\
j\not = 0 , \text{ odd} }}P ^\mu _\mp   (B_j) \bigr] \bigl[\prod
_{\substack{-m\leq j \leq n\\ j\not = 0 , \text{ even} }}P ^\mu _\pm
(B_j)  \bigr] /\EE(\ell _\pm) \,.
\end{multline*}
This identity implies that
  under $\Pp^\mu
( \cdot \,|\, \g_0 \in \Ww_\pm)$  the random paths $\g_i$, $i\in
\ZZ$, are independent,  the paths $\{\g_{2i} \}_{i\in \ZZ \setminus
\{0\}}$ have common law $P^\mu _\pm $ while the paths $\{\g_{2i+1}
\}_{i\in \ZZ }$ have common law $P^\mu_\mp $ and that  the path
$\g_0$ has law
\begin{equation}
\Pp^\mu \bigl(\g_0 \in A \,|\, \g_0 \in \Ww _\pm \bigr)\,=\,
\frac{\EE_{ P^\mu _\pm }\bigl(\,\ell (\g) \chi ( \g \in A)\,
\bigr)}{ \EE_{P^\mu_\pm} \bigl( \ell (\g) \bigr)}= \frac{\EE_{ P^\mu
_\pm }\bigl(\,\ell (\g) \chi ( \g \in A)\, \bigr)}{ \EE(\ell_\pm)}
\,.
\end{equation}


Finally, we claim that under  $\Pp^\mu \bigl(\cdot \,|\, \g_0 \in
\Ww _\pm \bigr)$ $x_1$ has probability density function on
$[0,\infty)$  given by $(1-F_\pm (x))/\EE(\ell_\pm) $ where
$$ F_\pm (x)= \PP  ( \ell_\pm  \leq x )\,, \qquad x \geq 0 \,.
$$
Indeed, given $x\geq 0$, due to  (\ref{spaghetti}) and
(\ref{parto1})  we can write
\begin{multline*}
\Pp^\mu \bigl( x_1 \geq x \,|\, \g_0 \in \Ww _\pm \bigr)=\Pp^\mu
\bigl( x_1 \geq x \,,\, \g_0 \in \Ww _\pm \bigr)/ \Pp ^\mu (\g_0 \in
\Ww _\pm )=\\ \EE_{\Pp ^\mu _{0,\pm} } \left( \int _0 ^{x_1} \chi(
x_1 -t \geq x ) dt \right) / \EE(\ell _\pm)= \EE\left( (\ell_\pm -x
) \chi(\ell_\pm \geq x)\right ) /\EE(\ell_\pm )=\\
 \EE\left( \int _x
^\infty \chi (\ell_\pm \geq z ) dz\right) /\EE(\ell_\pm) =\left(
\int _x ^\infty  \PP (\ell _\pm \geq  z) dz \right) / \EE(\ell_\pm )
\,.
\end{multline*}

In particular, we point out that due to (\ref{vito1}) and
(\ref{parto1})  under $\Pp^\mu$ the random variable $x_1$ has
probability density on $[0,\infty)$ given by $\left(2-F_+ (x)-
F_-(x) \right) /\EE(\ell) $.

\subsection{The statistics of $h$--slopes}

We can finally prove the main theorem of this section:

 \begin{theo}\label{napoleone}
Let $B$ be the  Brownian motion with  law $\PP^\mu_0$   and let
 $\bigl(\mathfrak{m}_i \bigr)_{i\in \ZZ}$ be the sequence of points of  $h$--extrema of $B$, increasingly
 ordered, with $ \mathfrak{m}_0\leq 0 < \mathfrak{m}_1$.
 For each $i \in \ZZ$ define the $h$--slope $\g_i$ as
 $$
 \g_i =\left( B_t - B_{\mathfrak{m}_i} \, :\, 0 \leq t \leq
 \mathfrak{m}_{i+1}-\mathfrak{m}_i\right)\,.
 $$
Then the marked point process \begin{equation}\label{violino}\left\{
(\mathfrak{m}_i, \g_i)\,:\, i \in \ZZ\right\} \end{equation}has law
$\Pp^\mu $.
\end{theo}

In order to prove the above result we need to   further elucidate
the relation between $\Pp^\mu$ and its Palm distribution
$\Pp^\mu_0$, benefiting of the Palm--Khinchin theory. To this aim we
need the ergodicity of  $\Pp^\mu _0$:
\begin{lemma}\label{saltelli}
The probability measure $\Pp^\mu_0$ is $\theta$--ergodic, i.e. if
$A\subset \Nn $ is a measurable  set such that $\Pp ^\mu _0 (A\Delta
\theta A)=0$ then $\Pp^\mu _0(A)\in\{0,1\}$.
\end{lemma}
\begin{proof}
Let us suppose that  $A\subset \Nn $ is a Borel set with $\Pp^\mu _0
(A\Delta \theta A)=0$.  Due to the characterization  of the
$\s$--algebra of measurable sets in $\Nn $ given by the bijective
and both ways measurable map (\ref{pmassa}) and since $\Pp ^\mu_0$
is concentrated on $\Nn _0$, for each $\e>0$ we can find a
measurable set $A_\e \subset \Nn_0 $  and an integer $k=k(\e)$ such
that $A_\e$ depends only on the random variables $ \g_i$ with
$-k\leq i \leq k$ and $\Pp^\mu _0 (A\D A_\e) \leq \e$. Since
$$
 \Pp^\mu _0(A\Delta A_\e)= \frac{1}{2} \Pp^\mu _{0,+}  ( A \D A_\e) +
\frac{1}{2} \Pp^\mu_{0,-}( A \D A_\e)
$$
we can conclude that
 \begin{equation}\label{vegetale1}
 \Pp^\mu_{0,+}  ( A \D A_\e) \leq 2\e, \qquad  \Pp^\mu_{0,-}( A \D
 A_\e)\leq 2 \e\,.
 \end{equation}
Since $\Pp^\mu_0$ is $\theta$--invariant and $\Pp _0 ^\mu  ( A \D
\theta A)=0$, we get for each positive integer $n$ that $ \Pp^\mu _0
(A\Delta \theta ^n A) =0$ and therefore
$$
 \Pp^\mu_0  (A\Delta \theta ^n A_\e)=
\Pp^\mu_0 (\theta ^n A\Delta \theta ^n A_\e) = \Pp^\mu _0 \bigl(
\theta ^n (A \D A_\e) \bigr)= \Pp^\mu _0 (A\D A_\e)\leq \e. $$ This
implies that
\begin{equation}\label{vegetale2}
\Pp^\mu _0 (A_\e \Delta \theta ^n A_\e )\leq \Pp^\mu _0 (A_\e \D A)+
\Pp ^\mu _0 (A \Delta \theta ^n A_\e )     \leq 2\e\,.
\end{equation}
Let us now write  $o(1)$ for   any quantity which goes to $0$ as $\e
\downarrow 0$. We note that  for $n$ large enough and even it holds
\begin{align}
& \Pp^\mu_{0,+}   ( A_\e \cap \theta ^n A_\e)= \Pp^\mu_{0,+}
(A_\e)^2=\Pp ^\mu_{0,+} (A)^2+o(1)\,, \label{vegetale3}\\
& \Pp ^\mu_{0,-}  ( A_\e \cap \theta ^n A_\e)= \Pp ^\mu_{0,-}
(A_\e)^2= \Pp ^\mu_{0,-} (A)^2+o(1)\,;\label{vegetale4}
\end{align}
while  for $n$ large enough and odd it holds
\begin{equation}\label{vegetale5}
\Pp ^\mu_{0,+} ( A_\e \cap \theta ^n A_\e)=\Pp^\mu_{0,-} ( A_\e \cap
\theta ^n A_\e)=\Pp^\mu_{0,+} (A_\e)\Pp ^\mu_{0,-}(A_\e)= \Pp
^\mu_{0,+} (A)\Pp^\mu_{0,-} (A)+o(1)\,.
\end{equation}
Let $a:=\Pp ^\mu_{0,+}(A)$ and  $b := \Pp ^\mu_{0,-}(A)$.  Due to
(\ref{vegetale1}),...,(\ref{vegetale5}) we can conclude that
\begin{equation} \begin{split} &  o(1)=
\Pp ^\mu_0 (A_\e \Delta \theta ^n A_\e )= \Pp^\mu _0(A_\e)+
\Pp^\mu_0 (\theta ^n A_\e)- 2 \Pp^\mu _0 (A_\e \cap \theta ^n A_\e)=\\
& 2 \Pp ^\mu_0 (A_\e) - 2 \Pp _0^\mu(A_\e \cap \theta ^n A_\e ) =\\
 &  \Pp^\mu_{0,+}  (A_\e) + \Pp^\mu_{0,-} (A_\e) -\Bigl[
\Pp^\mu_{0,+} (A_\e \cap \theta ^n A_\e) +
 \Pp^\mu_{0,-}  (A_\e \cap \theta ^n A_\e)\bigr] =
  \\
& \begin{cases} a+b - (a^2+b^2) + o(1) & \text{ if $n$ is even and
large}\,,
\\
a+b - 2 ab+o(1) & \text{ if $n$ is odd and large}\,.
\end{cases}
\end{split}
\end{equation}
We conclude that
$$
 a+b - (a^2+b^2)  =o(1)\,,\qquad a+b - 2 ab=o(1)\,.
$$
hence, by subtraction, $(a-b)^2=o(1)$, i.e. $a=b+o(1)$.  It is
simple to check that there
 are only two possible cases: (i)
$a=o(1)$ and $b=o(1)$, (ii) $a=1+o(1)$ and $b=1+o(1)$. Since
$\Pp^\mu_0(A)=(a+b)/2$,  in the first case we get $\Pp^\mu _0 (A)=
o(1)$ while in the latter $\Pp^\mu _0(A)= 1+o(1)$. Due to the
arbitrary of $\e$ we conclude that $\Pp^\mu _0(A)\in \{0,1\}$.

\end{proof}
Let us now introduce the space  $\Nn_*$ given by the counting
measures $\xi= \sum_{j\in J } n_j \d_{ (x^{(j)}, \g^{(j)} )}$ on
$\RR\times \Ww$, where $n_j \in \NN$ and  the set $\{\bigl(x^{(j)},
\g^{(j)} \bigr)_{j\in J} \}$ has finite intersection with sets of
the form $[a,b]\times \Ww$. Note that if $n_j=1$ for all $j\in J$ we
can identity $\xi$ with its support. Hence we can think of $\Nn$ as
a subset of $\Nn_*$.

As discussed in \cite{FKAS}[Section 1.1.5] one defines on  $\Nn _*$
a suitable metric $d _{\Nn _*}$ such that i) $\Nn_*$ is a Polish
space, ii) $\Nn $ is a Borel subset of $\Nn _*$ and iii) the
$\sigma$--algebra of Borel subsets of $\Nn _*$ is generated by the
sets
$$
\left\{ \xi \in \Nn _*\,:\; \xi (A\times B)= j \right\}\,, \qquad
A\subset \RR  \text{ Borel}, \; B \subset \Ww \text{ Borel},\; j \in
\NN \,.
$$
In particular, the $\s$--algebra of measurable subsets of $\Nn $
introduced above coincides with the $\s$--algebra of Borel subsets
of $\Nn $ and  we can think of $\Pp^\mu, \Pp^\mu_0, \Pp^\mu_{0,\pm}$
as Borel probability measures on $\Nn_*$ concentrated on $\Nn $.

  Due to
Lemma \ref{saltelli} and \cite{FKAS}[Theorem 1.3.13] we get
\begin{coro}\label{balocco} Given  $\mu\not =0$,
the probability measure  $T_{t_0} \Pp_0^\mu $ weakly converges to
$\Pp^\mu  $ as $t_0\downarrow -\infty$, i.e. for any continuous
bounded function $f:\Nn_*\rightarrow \RR$ it holds
\begin{equation}\label{pomodoro100}
\lim _{t_0\downarrow -\infty} \EE_{ T_{t_0} \Pp^\mu _0} (f) =
\EE_{\Pp^\mu} (f)\,.
\end{equation}
\end{coro}

Let $\nu$ be   the intensity measure   associated to $ \Pp ^\mu$,
i.e. $\nu$ is  the probability measure on $\RR \times \Ww$ such that
$$ \nu (A\times B) = \EE_{\Pp ^\mu } \bigl( \xi (A\times B) \bigr),
\qquad A\subset \RR \text{ Borel}, \; B \subset \Ww \text{
Borel}\,.$$ Then due to  \cite{FKAS}[Theorem 1.1.16], the weak
convergence of $ T_{t_0} \Pp _0 ^\mu$ to $\Pp ^\mu$ stated in
Corollary \ref{balocco} is equivalent to the following fact: given a
finite family  $X_1, X_2, \dots , X_k$ of disjoint sets
$$ X_j = [a_j,b_j) \times L_j\,, \qquad a_i,b_j \in \RR\,, \qquad
L_j \subset \Ww \text{ Borel}\,,
$$
satisfying $$\nu (\partial X_j ) =0\,, \qquad j=1,\dots, k \,,
$$
where $\partial X $ denotes the boundary of $X$, it holds
\begin{multline*}
\lim_{t_0 \downarrow -\infty} T_{t_0} \Pp^\mu_0 \left( \xi
(X_1)=j_1, \xi (X_2)=j_2, \dots , \xi (X_k) = j_k \right) = \\
\Pp ^\mu \left( \xi (X_1)=j_1, \xi (X_2)=j_2, \dots , \xi (X_k) =
j_k \right)
\end{multline*} for all $j_1, j_2, \dots, j_k \in \NN$.

\smallskip

We have now the main tools in order to prove Theorem
\ref{napoleone}.

\smallskip

 \begin{figure}[!ht]
    \begin{center}
       \psfrag{a}[l][l]{$\sigma_0$}
       \psfrag{b}[l][l]{$\tau_0$}
       \psfrag{c}[l][l]{$\sigma_1$}
       \psfrag{d}[l][l]{$\tau_1$}
       \psfrag{e}[l][l]{$\sigma_2$}
       \psfrag{f}[l][l]{$\tau_2$}
      \psfrag{g}[l][l]{$h$}
       \psfrag{z}[l][l]{$t_0$}
      \psfrag{m}[l][l]{$\b_1$}
      \psfrag{n}[l][l]{$\b_0$}
      \psfrag{p}[l][l]{$\b_2$}
       \includegraphics[width=12cm]{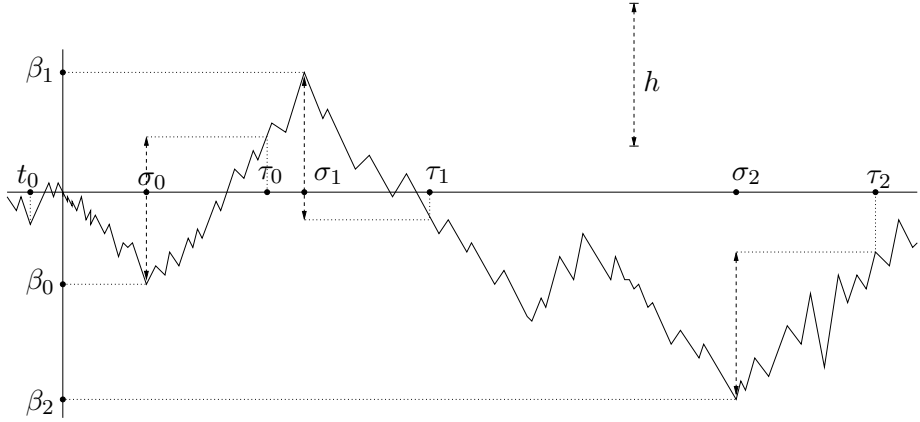}
      \caption{The sequence $\sigma_n$, $\tau_n$.}
    \label{tagliare}
    \end{center}
  \end{figure}

\noindent {\sl  Proof of Theorem \ref{napoleone}}. Let $B$ be a
two--sided Brownian motion with law $\PP^\mu_0$.  Set $\t_{-1}= t_0$
and define the random variables
$\t_n, \b_n, \s_n$ inductively on $n\in \NN$ as follows (see figure \ref{tagliare}):\\
For  $n$  even set
\begin{align}
& \t_n= \min\left \{ t \geq \t_{n-1} \,: \, B_t = \min_{\t_{n-1}
\leq
s \leq t } (B_s) +h \right\}\,,\\
& \b_n = \min \left\{ B_s \,:\, \t_{n-1}\leq s \leq \t_n \right\}\,,
\\
& \s_n = \max \left\{ s\,:\,\t_{n-1}\leq s \leq \t_n ,\,
B_s=\b_n\right\}\,;
\end{align}
for $n$ odd set
\begin{align}
& \t_n= \min\left \{ t \geq \t_{n-1} \,: \, B_t = \max_{\t_{n-1}
\leq
s \leq t } (B_s) - h \right\}\,,\\
& \b_n = \max  \left\{ B_s \,:\, \t_{n-1}\leq s \leq \t_n
\right\}\,,
\\
& \s_n = \max \left\{ s\,:\,\t_{n-1}\leq s \leq \t_n ,\,
B_s=\b_n\right\}\,.
\end{align}
Note that by construction $\s_n$ is a point of $h$--maximum for $n$
odd, while $\s_n$ is a point of $h$--minimum for $n\not =0 $ even.
Moreover, due to Lemma \ref{davide} and the strong Markov property
of Brownian motion at the Markov times $\t_n$, the slopes
$$
\bigl( B_{\s_n+s}-B_{\s_n} \,:\, 0 \leq s \leq \s_{n+1}-\s_n\bigr)
\qquad n \geq 1
$$
are independent, having  law $P^\mu_-$ if $n$ is odd and  law
$P^\mu_+$ if $n$ is even.

In what follows we will use the independent  random variables $X,V$
with the following laws:
  $X$ is   distributed as $\s_1-t_0$, i.e. $X$ is
distributed as $\bar \t + \bar \s ' $ where $\bar \t , \bar \s'$ are
independent copies of $\t,\s'$ defined in (\ref{esameinf}),
(\ref{esameinfbis}) respectively. $V$ is distributed as $\ell_+$,
i.e. as the length  $\ell(\g)$ of the random path $\g$  chosen with
law $P ^\mu _+$.

Given a realization  of the two--sided  Brownian motion $B$  with
law $\PP_0^\mu$, let  $\xi(B)$ be   the associated  marked simple
point process defined in (\ref{violino}), while let  $\xi$  denote a
generic element of $\Nn$. Fix
 a
finite family  $X_1, X_2, \dots , X_k$ of disjoint sets $ X_j =
[a_j,b_j) \times L_j$, with  $a_i,b_j \in \RR$ and $L_j\subset \Ww $
Borel,  and consider the event
$$
\Aa :=\left\{ \xi \,:\, \xi (X_1)=j_1, \; \xi (X_2)=j_2, \dots , \xi
(X_k)=j_k \right\}
$$
for given $j_1, j_2,  \dots, j_k \in \NN$. Finally, set $$a:=\min
\{a_1, a_2, \dots, a_k\}\,. $$
 Due to the discussion after Corollary
\ref{balocco}, we only need  to show that
\begin{equation}
\PP _0 ^{\mu} \bigl( \xi(B) \in \Aa \bigr)= \Pp^\mu (\Aa) \,.
\end{equation}
To this aim,  we set
$$
g(u) =  T_u \Pp ^\mu _{0,-} (\Aa
 )
$$
and restrict in what follows to the case  $t_0<a$. Then
 our initial
considerations  imply that
\begin{equation}
\PP_0^\mu \bigl( \xi(B)\in \Aa , \s_1 <a \bigr)= \EE \bigl(
g(t_0+X), t_0 +X< a \bigr)
\end{equation}
and therefore
\begin{multline}\label{marzo}
\left| \PP_0^\mu \bigl( \xi(B) \in \Aa \bigr) - \EE g( t_0 + X )
\right|=\\
 \left| \PP_0^\mu \left( \xi(B)\in \Aa , \; \s_1 \geq a\right)
- \EE\left( g( t_0 + X ),\; t_0 +X \geq a \right) \right|  \leq 2
P
( t_0 + X \geq a)\,.
\end{multline}
In what follows, we will frequently apply  the above  argument in
order to get estimates from above  without explicit mention.

\smallskip

Let us consider the probability measure $T_{t_0} \Pp ^\mu _0$. By
definition
\begin{equation*} 
T_{t_0}  \Pp ^\mu _0 ( \Aa ) = \frac{1}{2} T_{t_0} \Pp ^\mu _{0,+} (
\Aa ) + \frac{1}{2}T_{t_0} \Pp ^\mu _{0,-} ( \Aa )\,,
\end{equation*}
while
\begin{equation}\label{marzo2}
\left| T_{t_0} \Pp ^\mu _{0,+} (\Aa )- \EE g(t_0+V) \right|\leq 2 P(
t_0 + V\geq a )\,.
\end{equation}
Hence we can estimate
\begin{equation}\label{marzo3}
\left| T_{t_0}  \Pp ^\mu _0 (\Aa) -  \EE  g( t_0 + V )/2 -g (t_0) /2
\right| \leq 2  P( t_0 + V\geq a )\,.
\end{equation}
Due to (\ref{marzo}), (\ref{marzo3}) and Corollary \ref{balocco}, in
order to prove the theorem it is enough to show that
\begin{equation}\label{marzo4}
\lim _{t_0 \downarrow -\infty} \left|\EE g( t_0 + X ) -  \EE  g( t_0
+ V )/2 -g (t_0) /2\right| =0 \,.
\end{equation}
We will derive from the local central limit theorem that, given a
generic positive random variable $W$ having a (bounded) probability
density and bounded third moment, it holds
\begin{equation}\label{marzo5}
\lim _{t_0 \downarrow -\infty} \left|\EE g( t_0 + W)- g(t_0)
\right|=0\,.
\end{equation}
Due to Lemma \ref{uffina} below,  this result allows to derive
(\ref{marzo4}). In order to prove (\ref{marzo5}) define $S_n$ as the
sum of $n$ independent copies of the random variable $\ell $
introduced in Proposition \ref{calcoli}. Moreover, let $S_n$ and $W$
be independent. Then
\begin{align*}
& \left | \EE g (t_0+  W+S_n) - \EE  g (t_0 + W) \right|  \leq 2 P( t_0+ W+ S_n \geq a )\,,\\
&  \left | \EE g (t_0+S_n) - \EE  g (t_0) \right|  \leq 2  P( t_0+
S_n \geq a )\,.
\end{align*}
Hence in order to prove (\ref{marzo5}) it is enough to prove that
$$
\lim _{n\uparrow \infty}\,\, \limsup _{t_0 \downarrow -\infty}
 \left |
 \EE g (t_0+  W+S_n)-  \EE g (t_0+S_n)  \right| =0\,.
 $$
In general, given a r.w. $Z$ we write $p_Z$ for its probability
density (if it exists). Moreover, we denote by $\|\cdot \|_1$ the norm
in $L^1(\RR,du)$.   
Setting  $\bar S_n = S_n - n \EE (\ell )$, we can bound
\begin{equation}\label{brahms1}
 \left |
 \EE g (t_0+  W+S_n)-  \EE g (t_0+S_n)  \right| \leq  \int _\RR  g (t_0 +
 u ) \bigl| p_{W+S_n} (u) - p_{S_n} (u) \bigr| du \leq {\| p_{W+\bar S_n} - p_{\bar S_n}
\|}_1 \,.
\end{equation}
 Since
\begin{align*}
&  p_{W+\bar S_n} (u)du = p_{(W+\bar S_n)/\sqrt{n} } (u/ \sqrt{n} )
d(u/\sqrt{n} )\,, \\
&  p_{\bar S_n} (u)du = p_{\bar S_n/\sqrt{n} } (u/ \sqrt{n} )
d(u/\sqrt{n} )\,,
\end{align*}
by a change of variables we can conclude that
\begin{equation}\label{brahms2}
\left |  \EE g (t_0+ W+S_n)-  \EE g (t_0+S_n)  \right| \leq  {
\|p_{(W+\bar S_n)/\sqrt{n} }- p_{\bar S_n}/\sqrt{n}  \|}_1\,.
\end{equation} Let $\Nn(u) = \exp\bigl ( - u^2 /(2\l )\bigr)/ \sqrt{ 2 \p \l }  $ be the
gaussian distribution with variance  $\l = Var (\ell)$. Due to
(\ref{brahms1}) and (\ref{brahms2}) in order to conclude we only
need to show that
\begin{align}
& \lim_{n\uparrow \infty}  \|p_{(W+\bar S_n)/\sqrt{n}
}-p_{W/\sqrt{n}}  * \Nn \|_1=0\,, \label{centrale1}\\
& \lim_{n\uparrow \infty} \| p_{W/\sqrt{n}}  * \Nn -\Nn    \|_1 =0\,, \label{centrale2}\\
& \lim_{n\uparrow \infty} \|\Nn - p_{\bar S_n/\sqrt{n} } \|_1=0\,,
\label{centrale3}
\end{align}
where $f*g$ denotes the convolution of $f$ and $ g$.

\smallskip

 We note that (\ref{centrale1}) follows from
(\ref{centrale3}) since $p_{(W+\bar S_n)/\sqrt{n} }= p _{W/\sqrt{n}}
* p_{\bar S_n /\sqrt{n}} $ and for generic functions $h, h', w$ in
$L^1 (\RR, du)$ it holds $\|h* w - h'
* w \|_1\leq \|h - h'\|_1 \|w\|_1$.
  The limit (\ref{centrale2}) follows from straightforward
computations while  (\ref{centrale3}) corresponds to   the
$L^1$--local central limit theorem for densities since $W$ has
bounded probability density and bounded third moment   (see
\cite{PR}[page 193] or Theorem 18 in   \cite{Pe}[Chapter VII] where
the boundedness of $p_W$ is required).

 \qed

\begin{lemma}\label{uffina}
The random variables $X$ and $V$ in the proof of Theorem
\ref{napoleone}  have bounded continuous probability densities.
Moreover, they have finite n--th moment for all $n \in \NN$.
\end{lemma}
\begin{proof}
Due to Theorem 3 in \cite{F}[Section XV.3] in order to prove that
$X$ and $V$ have bounded continuous probability densities it is
enough that the associated Fourier transforms are in $L^1(\RR, dx)$.
Since $X= (\t-\s) + \s+ \s'$ and $\ell_+=(\t-\s)+\s'$ where the
random variables $\t-\s$, $\s$ and $\s'$ are independent, it is
enough to prove that the Fourier transform of $\t-\s$ is integrable.
To this aim we observe that due to Lemma \ref{davide}  the
expectation  $\EE_0^\mu \bigl(e^{-\a (\t-\s)}\bigr)$ is  finite for
$\a
> -\mu^2 /2$. This implies that the complex Laplace transform
$$\CC \ni \a \rightarrow \EE_0^\mu \bigl(e^{-\a (\t-\s)}\bigr)\in \CC
$$
is  well-defined (i.e. the integrand is integrable) and analytic on
the complex halfplane $\Re (\a)
> -\mu^2 /2 $. Indeed, integrability is stated in  Section
2.2 of \cite{D1} and  analyticity is stated in  Satz 1 (Proposition
1) in Section 3.2 of \cite{D1}. We point out that in \cite{D1} the
author considers  the  complex Laplace transform of functions, but
all arguments and results can be easily  extended to the complex
Laplace transform of probability measures.
\\
In particular, we get  that the   Fourier
transform  $\t-\s$ is given by
 $$
 \EE_0^\mu \left( \exp \left( -i a (\t-\s)\right) \right)
=\frac{\sqrt{2a i +\mu^2 }}{\mu} \frac{  \sinh (\mu h ) }{
\sinh\left(\sqrt{2 a i +\mu ^2 } h \right) }\,,\qquad a\in \RR\,,
$$
where the square--root is defined by analytic extension  as $\sqrt{r
e^{i\theta} }= \sqrt{r} e^{i\theta /2} $ on the simply connected set
$\{r e^{i\theta} : r\geq 0, \theta \in (-\p, \p)\}$. From the above
expression, one easily derives  the integrability of the  above
Fourier transform.

\smallskip

Since the Laplace transforms (\ref{lapla1}) and  (\ref{lapla2}) are
analytic in the origin, it follows from \cite{F}[Section XIII.2]
that  $\s$, $\t-\s$ and $\s'$  have  finite n--th moments for all $n
\in \NN$. Hence, the same holds for $X$ and $V$.

\end{proof}

\section{Comparison with the RG--approach}\label{rengroup}

In this section we give some comments on   the results concerning
the $h$--extrema of drifted Brownian motion obtained in \cite{DFM}
via the non rigorous  Real Space Renormalization Group (RSRG) method
and applied for the analysis of 1d random walks in random
environments. We present the results obtained in \cite{DFM} in the
formalism of Sinai's random walk, keeping the discussion at a non
rigorous level.

Start with a sequence of i.i.d. random variables $\{\o_x\}_{x\in
\ZZ}$ such that $\o _x \in (0,1)$ and
$$
A:=\EE \left[ \log \frac{1-\o_0}{\o_0} \right] \in \RR
\setminus\{0\}\,,\qquad  \text{Var} \left[ \log \frac{1-\o_0}{\o_0}
\right] =:  2\s  \in (0,\infty).
$$
Defining  $ \d:= A/(2\s)$, the random variable  $\log \o_x
/(1-\o_x)$ corresponds then  to the random variable $f$ in (27) of
\cite{DFM}.
Without loss of generality we set $\s=1$ as in \cite{DFM}, thus
implying that
 $A= 2 \d$.

 The associated Sinai's random walk is the nearest
neighbor random walk on $\ZZ$ where $\o_x$ is the probability to
jump from $x$ to $x+1$ and $1- \o_{x-1}$ is the probability to jump
from $x$ to $x-1$. Consider the function $V: \RR \rightarrow \RR $
defined on $\ZZ$ as
$$
V(x)=
\begin{cases}
\sum _{i=0} ^{x-1} \log \frac{1-\o_x}{\o_x}\,, &\text{ if } x\geq
1\, , x\in \ZZ\,, \\
0 & \text{ if } x=0\,, \\
- \sum _{i=x} ^{-1 } \log \frac{1-\o_x}{\o_x} & \text{ if } x <0 \,,
x\in \ZZ\,,
\end{cases}
$$
and extended to all $\RR$ by linear interpolation.  Morally, the
above Sinai's random walk is well described by a diffusion in the
potential $V$. In \cite{DFM} the authors obtain   results on the
statistics of the  $\G$--extrema of $V$ taking the limits $\G
\uparrow \infty $,  $\d \downarrow 0 $ with $\G \d$ fixed. In what
follows, we   show    the link between their results and our
analysis of the statistics of $h$--extrema of drifted Brownian
motion.

  By  the Central Limit Theorem applied to
$ V(x) -2 \d   x$,
 one  concludes that for $\G$ large
$$
 \frac{V (x \G^2)}{ \sqrt{2}  \G }\sim B^*_x + \sqrt{2}\d \G x \, ,
 \qquad x \in \RR\,,
 $$
where $B^*$ is the standard two--sided   Brownian motion (i.e. $B^*$
has law $\PP_0$). If we set \begin{equation}\label{neffa}
 \mu =-
\sqrt{2} \d \G   \end{equation} and consider the limits $\G \uparrow
\infty $ and $\d \downarrow 0 $ with $\mu$ fixed we get that the
rescaled potential $V$ is well approximated by the Brownian motion
$B$ with law $\PP_0 ^{\mu}$. In particular, for $\G$ large one
expects that the ordered family  $ \{( x_k, V(x_k) )\}_{k\in \ZZ}$
of $\G$--extrema of $V$ is  well approximated by family $\left\{
(\G^2 \mathfrak{m}_k, \sqrt{2} \G B_{\mathfrak{m}_k} )\right\}_{k\in
\ZZ}$ where $\{ \mathfrak{m}_k \} _{k\in \ZZ}$ is the ordered
sequence of points of $h$--extrema of $B$ with $h=1/\sqrt{2}$ (we
follow the convention that $x_0 \leq 0 < x_1$).
Setting as in \cite{DFM}
$$ \zeta:=   \bigl | V(x_{k+1} ) - V(x_k) \bigr|- \G    \,,\qquad    l:=x_{k+1}-x_k
 \,,
 $$
morally we get \begin{equation}\label{arrosticini}
 \zeta /\G \sim \sqrt{2} \bigl | B
_{\mathfrak{m}_{k+1} }- B_{\mathfrak{m}_k} \bigr| -1\,, \qquad
l/\G^2 \ \sim\mathfrak{m}_{k+1}-\mathfrak{m}_k\,.
\end{equation}
In \cite{DFM} the authors
  write $P^+( \zeta , l) d\zeta d l$, $P^-(\zeta , l) d\zeta d l$
for the joint probability density  of the random variables $\zeta,
l$ if $x_k$ is a point of  $\G$--minimum or  a point $\G$--maximum
respectively, they set  $ P_\G ^\pm  (\zeta, p) := \int_0 ^\infty
e^{-l p } P_{\G}(\zeta, l ) d l  $ and
  derive via the RSRG method the limiting form of $P_\G ^\pm  (\zeta,
  p)$
  (note that in \cite{DFM} the authors erroneously do
 not distinguish the $\G$--slope covering the origin from the other
 $\G$--slopes, but in order to have a correct result one must
 take $k\not =0$).

 It is simple to check that  all the computations obtained in
 \cite{DFM}[Section II.C.2] equal the results obtained in
 Proposition \ref{calcoli} by approximating the random variables
$\zeta, l$ with distribution
  $P^\pm ( \zeta, l ) d \zeta d l$ by means of  the random variables
$ \sqrt{2} \G \zeta_\pm, \G^2 \ell_\pm$ of  Proposition
\ref{calcoli}, where $\mu= - \sqrt{2} \d \G $, $h=1/\sqrt{2}$. This
confirms (\ref{arrosticini}).

\section{Proof of Lemma \ref{davide} via  fluctuation  theory}\label{pierpa}
Given a path $f\in C([0,\infty),\RR)$,   define the hitting time of
$f$  at $x$  as
\begin{equation}
T_x(f)=\inf \left\{t>0\,:\, f_t = x\right\},\qquad x\in \RR \, .
\end{equation}
Consider the process $\left\{ B_t,t\geq 0 \right\}$ carrying the law
$\PP ^\mu _0$ and define (see figure \ref{coniglio1})
\begin{equation}\label{salgo}
\begin{cases}
b_t = \min\left\{B_s\,:\, 0\leq s \leq t \right\}\\
L_t =-b_t,\\
Y_t= B_t -b_t.
\end{cases}
\end{equation}
The process $Y$ is the so called one--sided drifted BM reflected at
its last infimum. It has the following properties:
\begin{figure}[!ht]
    \begin{center}
       \psfrag{a}[l][l]{$b_t$}
       \psfrag{b}[l][l]{$t$}
       \psfrag{c}[l][l]{$Y_t$}
      \psfrag{d}[l][l]{$h$}
       \includegraphics[width=7cm]{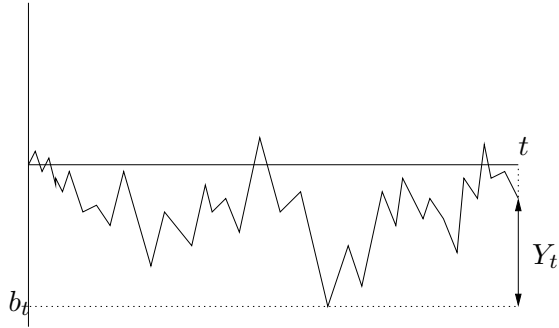}
      \caption{The process $Y_t$.}
    \label{coniglio1}
    \end{center}
  \end{figure}

\begin{lemma} \label{japan}\cite{RW2}[Lemma VI.55.1]\\
The process $Y=\left\{Y_t,\,t\geq 0 \right\}$ is a diffusion and
$L=\left\{L_t\,,\,t\geq 0\right\}$
 is a  local time
of $Y$ at $0$. The  transition density function of the process  {\sl
$Y$ stopped at $0$}, i.e.  $\left\{ Y_{t\wedge T_0}\,,\,t\geq
0\right\} $, is
\begin{equation}\label{aquila}
\bar p _t   (x,y)= (2\p t ) ^{-1/2 } e^{\mu (x-y)-\frac{1}{2} \mu ^2
t }\left[
 e^{-(y-x)^2 /2t }- e^{-(y+x)^2 /2t }\right]\,,\qquad x,y > 0\,.
\end{equation}
The entrance law $n _t  $, $t>0$, associated with the  excursions of
$Y$ from $0$ w.r.t the time  $L$ is given by $n_t  (dy) = n_t
(y)dy$, where:
\begin{equation}\label{roma}
n_t  (y) = 2y (2\pi t^3 )^{-1/2} \exp \left[ -(y+\mu t)^2/2t
\right],\qquad  y > 0.
 \end{equation}
\end{lemma}
Let us comment the above result and fix some notation
 (for a general treatment of excursion theory see for example \cite{RW2}[Chapter VI]
  and \cite{B}[Chapter IV]).

We denote by  $U$ the space of excursions from $0$, i.e. continuous
functions $f : [0,\infty)\rightarrow \RR$ satisfying the coffin
condition
$$ f(t) = f(H) = 0,\qquad \forall t\geq H,$$
where $H$ is the life time of $f$, namely
\begin{equation}
 H= H(f) = T_0(f)  \in [0,\infty].
 \end{equation}
The excursion space $U$ is endowed of the smallest $\s$--algebra
which makes each evaluation map $f\rightarrow f(t)$ measurable. One
can prove that this $\s$--algebra coincides with the Borel
$\s$--algebra  of the space $U$ endowed of the Shohorod metric.

Write $\g_t$ for the right continuous inverse of $L_t$, namely
$$ \g_t =\inf\left \{ u>0\,:\, L_u >t \right\} = \inf \{ u >0: \min _{0\leq s \leq u} B_s < -t \}  \,,$$
and  define the excursion $e_t \in U$, $t>0$, as
$$ e_t (s) = \begin{cases}
Y(\g_{t-} +s )& \text{ if  } \,0\leq s\leq \g_t-\g_{t-},\\
0 &\text{ if } \, s\geq   \g_t-\g_{t-}.
\end{cases}
$$
Then
 the random  point process $\nu$ of excursions  of $Y$ from $0$  is defined
 as
 $$\nu =\left\{ (t, e_t): \,t>0  ,\, \g_t \not = \g_{t-} \right\}\,.$$
In what follows, we will often identify the random discrete set $\nu
$  with the
 random measure $\sum _{t>0\,:\, \g_t \not =\g_{t-} } \d _{(t, e_t) }$ on $(0,\infty) \times U$.

 Decomposing $U$ as
 $U=U_\infty\cup U_0$, where
$$
 U_\infty=\{f\in U\,:\, H(f)=\infty\},\qquad U_0 =\{f\in U\,:\,
 H(f)<\infty\}\,,
$$
  It\^o Theorem \cite{RW2}[ Theorem VI.47.6] states that there exists a  $\s$--finite measure $n$ on $U$ (called
It\^o measure) with $n(U_\infty)<\infty$ such  that,
 if $\nu'$ is a Poisson point process on $(0,\infty)\times U$ with intensity measure $dt\times n$
  and if
\begin{equation}\label{zitta}
\zeta=\inf\left\{t>0\,:\, \nu'\left((0,t]\times U_\infty
\right)>0\right\},
\end{equation}
then the point process $\nu$ under   $\PP_0^\mu$ has the same law of
$\nu'\large{|} _{(0,\zeta]\times U}$:
\begin{equation}\label{pizza}
\nu \sim
 \nu'\large{|} _{(0,\zeta]\times U}.
\end{equation}
Here and in what follows,
given a measurable space $X$  with measure $m$ and a measurable
subset $A\subset X$ we denote $m\large{|}_ A$ the measure on $X$
such that
$$
 m\large{|} _A (B) =m(A\cap B),\qquad \forall B\subset X\text{ measurable}.
 $$

Given $t>0$ the entrance law $n_t(dy)$, with support in
$(0,\infty)$, is defined as \begin{equation}\label{saratoga}
 n_t
(dy) = n \left(\left\{ f\,:\, H(f)>t,\; f_t\in dy\right\} \right).
\end{equation}
Since the  process $Y$ (starting at $0$) defined via (\ref{salgo})
is Markov and
 visits each $y\geq 0$ a.s., the  definition of  the process $Y$
 starting at $y$ is obvious.
   Due to Lemma \ref{japan}, the process $Y$ starting at $y$ and stopped  at $0$ is a
    strong  Markov process with transition probability $\bar p_t (\cdot,\cdot)$.
    In what follows we denote its law by  $Q_y$.

Then, given $t>0$,  measurable subsets  $\Aa, \Cc\subset U$ with
$\Aa\in \s( f_s\,,\,0\leq s\leq t)$, it holds
\begin{equation}\label{rino1}
n\left(
 f\,:\, f\in \Aa    ,\, H(f)>t,\, \theta _t f  \in \Cc  \right) =\int _0 ^\infty
n\left( f\in \Aa,\, H(f)>t\,,f_t\in dy \right)    Q  _y (\Cc ),
\end{equation}
where $(\theta _t f )_s=f_{t+s}$.
 In particular,  due to (\ref{rino1})  the
transition density functions (\ref{aquila}) and the entrance laws
(\ref{roma})  determine univocally the It\^o measure $n$.


\bigskip

In order to get more information on the It\^o measure $n$ of the
point process of excursions of $Y$ from $0$,
 we give an
alternative probabilistic interpretation of the  transition density
$\bar p  (x,y)$.  To this aim  recall that
  Girsanov
formula implies that
\begin{equation}\label{aperol}
\EE ^\mu _x( g) =\EE _x (g Z_t ),\qquad Z_t = \exp\left\{ -\mu
(B_t-x)  -\mu^2 t/2 \right\},
\end{equation}
for each  $\Ff _t$--measurable function $g$, where $\Ff_t = \s
\left( B_s\,:\,0\leq s\leq t\right)$. Due to (\ref{aperol}), we get
for all $x,y,z,s,t>0$,
\begin{multline}\label{islandina}
 \PP _z ^\mu \left( B_{s+t}\in dy, \, T_0 ( \theta _s B) >t \,|\, B_s = x \right)=
\PP ^\mu _x \left( B_t \in dy , \,T_0 >t \right) =\\
e^{-\mu (y-x)  -\mu ^2 t/2 } \PP_x  \left ( B_t\in  dy\,\,
T_0>t\right)= \bar p_t ^\mu (x,y).
\end{multline}
In fact, the second identity follows from (\ref{aperol}) while the
last identity follows from (\ref{aquila})
 by computing
 $\PP_x \left ( B_t= y,\, T_0>t\right)$ via a reflection argument.
Hence,  given  $z>0$,  $\bar p_t (\cdot,\cdot)$  is the transition
density function of the process
 $\left( B_{t\wedge T_0 },\,t\geq 0 \right) $ under
$\PP ^\mu _z$,   whose law equals  $Q_z$.
 In particular,
(\ref{rino1}) can be reformulated as
\begin{equation}\label{rino2}
n\left(
 f\,:\, f\in \Aa    ,\, H(f)>t,\, \theta _t f  \in \Cc  \right) =\int _0 ^\infty
n\left( f\in \Aa,\, H(f)>t\,,f_t\in dy \right)
  \PP_y ^\mu  ( B_{\cdot\wedge T_0 }\in \Cc )\, .
\end{equation}
The above identity will be frequently used in what follows.

\bigskip
We point out that,  as stated in Theorem 1 of \cite{B}[Section
VII.1],  the content of Lemma \ref{japan} is valid in more
generality for spectrally positive L\'{e}vy processes s.t. the
origin is a regular point, i.e. real valued processes starting at
the origin with stationary independent increments, with no negative
jumps and returning to the origin at arbitrarily    small times.
Moreover, defining
$$ \tilde b_t :=  0 \wedge \inf
\{B_s : 0 \leq s \leq t \}\,,$$ it is simple to check that the
process $Y$ starting  at $x>0$ has the same law of the process
$\bigl(\tilde Y_t:= B_t - \tilde b _t\,,\, t\geq 0 \bigr)$, where
$B$ is chosen with law $\PP_x ^\mu$. This implies that  $\tilde Y_t
= B_t $ if $t< T_0(B )$, hence once gets again that $ \bar p ^\mu _t
(x,y) = \PP ^\mu_x (B_t \in dy, T_0 >t )$ as in (\ref{islandina}).

\bigskip

In order to state our results it is useful to fix  some further
notation. Given $h>0$ we denote by $U^{h,+}$ the family of
excursions with height at least $h$ and
 by $U^{h,-}$ the family of   excursions with
height less than $h$, namely
\begin{align}
&U^{ h,+ } =\left\{ f\in U\,:\, \sup _{s\geq 0} f_s \geq h \right\},\\
&U^{h,-}=\left\{ f\in U\,:\, \sup _{s\geq 0} f_s < h \right\}=
U\setminus U^{h,+}.
   \end{align}

\bigskip

One of the  main technical tools in order to extend the proof in
\cite{NP} to the drifted case is the following lemma, whose proof is
postpone to Section  \ref{primavera}.
\begin{lemma}\label{kuka}
If $\hat \a >0$, $\mu\not =0$,  then
\begin{align}
& n \bigl( U^{h,+}\bigr) =
\frac{ \mu e^{-\mu h }}{ \sinh (\mu h )} \label{kuka1}\, ,\\
& n\bigl( U^{h,-}\cap U_\infty\bigr) =0 \label{kuka4}\,,\\
& \int _{U^{h,-}} \left( 1-e^{-\a H(f)}
 \right) n(df) = \sqrt{2\hat\a} \coth \left(\sqrt{2\hat \a} h\right) -\mu \coth(\mu h ) \,, \label{kuka2} \\
 &  n \bigl( U^{h,+}\bigr) ^{-1} \int_U  e^{-\a T_h }\II_{T_h<T_0}  n(df) =
\frac{\sqrt{2\hat \a}}{\mu} \frac{  \sinh (\mu h ) }{
\sinh\left(\sqrt{2\hat \a} h \right) } \label{kuka3}\,.
\end{align}
\end{lemma}

Finally, we can prove Lemma \ref{davide}:

\smallskip

\noindent {\em Proof of Lemma \ref{davide}}. One can recover the
Brownian motion from the point process $\nu$ of excursions of $Y$
from $0$ by the formula
\begin{equation}\label{iaco}
B_t= -a+f(t-S),\qquad \text{ for } t\in [S, S+H(f)],
\end{equation}
which is valid for each couple $(a,f)\in \nu $  by setting
$$ S=\int_{ (0,a )\times U}   H(f') \nu(da',df').
$$
 It is convenient to associate to   $U^{ h,+ }$, $ U^{ h,-}$ the
 measures   $ \nu^*=\nu \large{|}_{[0,\infty)\times U^{ h,+ } }$,
   $ \nu_*=\nu \large{|}_{[0,\infty)\times U^{h,- } }$,
 $ n^*=   n|_{U^{h,+ }}$ and
 $n_*=   n|_{U^{h,-}}$. Moreover, we set
$$ a^*=\inf \left\{a>0: \,\exists f\in U^{h,+} \text{ with } (a,f)\in \nu \right\}.
$$
If $a^*$ is finite, let  $f^*$  be  the only excursion such that
$(a^*, f^*)\in \nu $.  Due to (\ref{zitta}), (\ref{pizza}) and
(\ref{kuka4})
$$
\PP_0^\mu (a^*>a  )= \PP_0^\mu  \left ( \nu^* \left((0,a]\times
U\right)=0\right)= \exp \left\{-a \, n^* ( U )\right\},
$$
therefore
 $a^*$ is an exponential variable with parameter $ n ^*(U)$ (in particular, $a^*$ is finite a.s.).
Due to the representation (\ref{iaco}), $\b =-a^*$. Together with
(\ref{kuka1}) this implies that $-\b$ is an exponential variable
with mean (\ref{lapla}). Moreover, (\ref{iaco}) implies that
\begin{equation}\label{ramnia1}
\s = \int_{ (0,a^* )\times U}    H(f') \nu(da',df')  =
 \int_{ (0,a^* )\times U}    H(f') \nu_* (da',df') .
 \end{equation}
 Due to the above expression and the representation (\ref{iaco}) , the trajectory
$\left(B_t\, ,\,0\leq t\leq \s \right)$ depends only on $\nu_*$ and
$a^*$, while the trajectory $\left(B_{\s+t}-\b\,,\, 0\leq t \leq
\t-\s \right)$ coincides with the excursion $f^*$ stopped when it
reaches level $h$. Since   $\nu_*$ and $a^*$ are independent from
$f^*$ we get the independence of the trajectories.

In order to prove (\ref{lapla1}) we observe that  $-\b = x$ means
that $a^*=x$. Therefore, conditioning  to $-\b =x$, it holds
$$\s =
\int_{ (0,x)\times U}    H(f') \nu_* (da',df'),
$$
thus implying that
\begin{equation}\label{miele}
 \EE_0^\mu\left[ \exp (-\a \s )\,\large{|}\, \b =-x \right]=
 \EE_0 ^\mu  \left(\exp\left\{
 -\a\int_{ (0,x)\times U}    H(f') \nu_* (da',df')  \right\} \right)\,.
 \end{equation}
Note that, in order to derive the above identity, we have used that
$\nu$ is the superposition of the independent point processes
$\nu^*$ and $\nu_*$.

We claim that
\begin{equation}\label{sale}
\EE_0 ^\mu  \left(\exp\left\{
 -\a\int_{ (0,x)\times U}    H(f') \nu_* (da',df')  \right\} \right)=\exp\left\{-x \int _U \left( 1-e^{-\a H(f)}
 \right) n_*(df) \right\} \,.\end{equation}
In order to prove this claim we
 note that, due to Ito Theorem and (\ref{kuka4}), the point process
 $\nu_*$ has the same distribution of the Poisson point process on
 $(0,\infty)\times U$ with intensity $dt \times n_*$. Hence, for
 $\a\geq 0$ the above  identity follows directly from  the exponential formula for
 Poisson point processes  \cite{B}[Section O.5]. Suppose now that $\a <0$ and $\hat \a   \geq 0$.  Given $m
  >0$ and $f \in U$ we define $H_m(f)$   as $-\infty$ if $H(f) \leq m$ and as $H(f)$ if
$H(f)>m$. Due to (\ref{saratoga}) and (\ref{rino2}), we get the
bound
$$
\int e^{-\a H_m (f) } n_*(df) \leq  e^{-\a m}  \int _0 ^h  n _m (dy
)  \EE^\mu _y \left ( e^{-\a T_0 } \II _{T_0 < T_h } \right)
$$
where  the r.h.s. is finite due to the form of $n_m$ and identity
(\ref{leone1}). This allows to conclude that  the integral $ \int
e^{-\a H_m (f) } n_*(df)$ is finite,  and therefore the same holds
for the smaller  integral $\int \bigl| 1- e^{-\a H_m (f) }
\bigr|n_*(df )$. This last property allows to apply again the
exponential formula for Poisson point process and to deduce that
\begin{equation*}
\EE_0 ^\mu  \left(\exp\left\{
 -\a\int_{ (0,x)\times U}    H_m (f') \nu_* (da',df')  \right\} \right)=\exp\left\{-x \int _U \left( 1-e^{-\a H_m(f)}
 \right) n_*(df) \right\}\,.
 \end{equation*}
Taking the limit $m \downarrow 0 $ and applying the Monotone
Convergence Theorem we derive (\ref{sale}) from the above identity.
Hence,   (\ref{lapla1}) follows from (\ref{kuka2}), (\ref{miele})
and (\ref{sale}), while
  trivially (\ref{lapla2}) follows from (\ref{lapla1}).

  Finally, in  order to prove (\ref{lapla3}), we observe that
$ \t-\s =T_h(f^* )$.  Since  the path $f^*$ has law $n^*/n^*( U ) $,
$$
\EE_0^\mu  \left( \exp \left( -\a(\t-\s)\right) \right) =n^*(U)^{-1}
\int _U n^* (df) e^{-\a T_h (f)}
$$
  and the thesis follows from (\ref{kuka3}).
\qed




\begin{rem}\label{piponebis}
As already remarked, the analogous of Lemma \ref{davide} (restricted
to $\a>0$)  has been proved for more general spectrally  one--sided
L\'{e}vy processes \cite{AKP}, \cite{Pi} and \cite{C}[Proposition
1], by means of more sophisticated arguments always based on
fluctuation theory, excursion theory and the analysis of the hitting
times of the process. We have given a self--contained and direct
proof based on simple computations, which will be useful also for
the proof of Theorem \ref{bonaparte}, but one can derive Lemma
\ref{davide} from the cited references as follows. The Laplace
exponent of the drifted BM with law $\PP_0^\mu$ is given by $ \psi
(\l)= \frac{1}{2}\l^2- \l \mu$, i.e. $\EE_0 ^{\mu} \bigl( e^{ \l
B_t}\bigr)= e^{t \psi (\l) }$ for $\l \in \RR$. Given $\a>0$ we
define the function $W^{(\a)}$ as
$$
W^{(\a)}=e^{\mu x}\left( e^{x \sqrt{2\hat \a} }-  e^{-x \sqrt{2\hat
\a} }\right)/\sqrt{2\hat \a } \,.$$ Then it is simple to check that
$$ \int_0 ^\infty e^{-\l x} W^{(\a)} (x) dx = \frac{1}{\psi (\l)-\a}
\,, \qquad \forall \l \geq \Phi (\a) \,,
$$
where the value $\Phi (\a) $ is defined as the largest root of $\psi
(\l) = \a$, i.e. $\Phi (\a) = \mu + \sqrt{ 2 \hat \a }$.  The
function $W^{(\a)}$ is related to the exit of the BM from a given
interval. More precisely, due to (\ref{leone2}), it holds
\begin{equation}\label{resistenza1000}
\EE _0 ^\mu ( e^{-\a T_y}  T_y < T_{-x}  )= W^{(\a)}
(x)/W^{(\a)}(x+y) \,, \qquad \forall x,y
>0\,.
\end{equation}
To the function $W^{(\a)}$ one associates the function $Z^{(\a)}$
given by
$$
Z^{(\a)} (x) = 1 + \a \int _0 ^x W^{(\a) } (z) dz=\frac{\a e^{\mu
x}}{\sqrt{2\hat \a} } \left( \frac{e^{\sqrt{2\hat \a}x}}{\mu +
\sqrt{2\hat \a}} -\frac{e^{-\sqrt{2\hat \a}x}}{\mu - \sqrt{2\hat
\a}}\right)
$$
Knowing the values of  $W^{(\a)}$ and $Z^{(\a)}$ one can compute the
expressions in Lemma \ref{davide} for $\a>0$ by applying for example
Proposition 1 in \cite{C}.
\end{rem}

\bigskip

%
%
%
\section{The behavior of the drifted Brownian motion near to an $h$--extremum}\label{autunno}

In this section we characterize the behavior of an $h$--slope not
covering the origin,  near to  its extremes.
 To this aim we recall the definition of the drifted
Brownian motion Doob--conditioned to hit $+\infty$ before $0$,
referring to \cite{B}[Section VII.2] and references therein for a
more detailful discussion. First, we write $W(x)$ for the function
$$ W(x):=W^{(0)}(x)=
\frac{e^{2x\mu} -1}{\mu}\,$$ ($W^{(\a)} $ has been defined in Remark
\ref{piponebis}).
Defining $\Phi (0)$ has the largest zero of $\psi(\l):=\l^2/2-\l \mu
$, i.e. $\Phi (0) := 0 \lor (2 \mu)  $, the function $W$ is a
positive increasing function with Laplace transform
$$
\int_0^\infty e^{-\l x} W(x) dx = \frac{1}{\psi(\l)} \,, \qquad
\forall \l > \Phi (0)\,,
$$
satisfying the identity  (see (\ref{resistenza1000}))
\begin{equation}\label{resistenza}
\PP _0 ^\mu ( T_y < T_{-x}  )= W(x)/W(x+y) \,, \qquad \forall x,y
>0\,.
\end{equation}
 Due to the above considerations, the
function $W$ is the so called {\em scale function} of the drifted
Brownian motion with law $\PP ^\mu_0$.

\smallskip

For each $x>0$ consider the new probability measure $\PP ^{\mu,
\uparrow}_x$ on the path space $C\bigl([0,\infty), \RR\bigr)$
characterized by the identity
\begin{equation}\label{jonny}
\PP ^{\mu , \uparrow } _x (\L) = \frac{1}{W(x)} \EE _x ^\mu \left(
W(X_t) , \L, t< T_0 \right)\,, \qquad \L \in \Ff_t\,,
\end{equation}
where $(X_t, t \geq 0)$ denotes a generic element of the path space
$C\bigl([0,\infty), \RR\bigr)$ and $\Ff_t:= \sigma \bigl\{ X_s:
0\leq s \leq t \bigr\}$. As discussed in \cite{B}[Section VII.3],
the above probability measure is well defined,  the weak limit
$\PP^{\mu, \uparrow } _0:= \lim _{x \downarrow 0} \PP ^{\mu,
\uparrow } _x $ exists and the process $\bigl( \PP^{\mu, \uparrow}
_x, x \geq 0 \bigr)$ is a Feller process, hence   strong Markov (we
point out that in \cite{B}[Section VII.3]  the above results are
proven in the Skohorod path space $D( [0,\infty), \RR )$, but one
can adapt the proofs to $C ([0,\infty), \RR)$).  As explained in
\cite{B}[Section VII.3], this process can be thought of as the
Brownian motion with drift $-\mu$ Doob--conditioned to hit $+\infty
$ before $0$. In the case of positive drift, i.e. $\mu <0$, this can
realized very easily  by observing that due to (\ref{resistenza})
\begin{multline*}
\PP _x ^\mu (T_0=\infty)= \PP_0^\mu(T_{-x} =\infty) =
\lim_{y\uparrow \infty} \PP_0^\mu (T_y< T_{-x} ) =\\
 \lim
_{y\uparrow \infty} \frac{W(x)}{W(x+y)}= 1- e^{2x\mu} = -\mu W(x)\,,
\qquad \forall x >0\,,
\end{multline*}
and that this identity together with   the Markov property implies
that
 $$
\PP _x ^{\mu } (\L| T_0 = \infty) = \PP _x ^{\mu,\uparrow} (\L) \,,
\qquad x>0, \L \in \Ff _t \,.
$$
For negative drift the event $\{T_0=\infty\}$ has zero probability,
and a more subtle discussion is necessary.

\begin{lemma}\label{black}
The process
$\PP^{\mu, \uparrow}_0$
 is a diffusion characterized by
the SDE
\begin{equation}\label{partigiani}
d X_t = dB_t +\mu \coth \left( \mu X _t \right) dt\,, \qquad
X_0=0\,,
\end{equation}
where $B_t$ is the standard Brownian motion.
\end{lemma}
\begin{proof}
As already discussed, the above process has continuous paths and it
is strong Markov, i.e. it is a diffusion.

Due to (\ref{jonny}) and (\ref{islandina}), given $x,y>0$,
\begin{multline}
q_t(x,y):= \PP _x ^{\mu,\uparrow } (X_t \in dy )=\frac{W(y)}{W(x)}
 P_x^{ \mu} \bigl(X_t \in dy , t< T_0 \bigr)= \frac{W(y) \bar p _t
 ^\mu (x,y) }{W(x)} =\\
 \frac{ \sinh (\mu y ) }{\sinh (\mu x) }
  \frac{ e^{-\frac{1}{2} \mu ^2 t }  }  { \sqrt{2\p t } }\left[
 e^{-(y-x)^2 /2t }- e^{-(y+x)^2 /2t }\right]
   \,.
\end{multline}
From the above expression, by direct computations one  derives that
\begin{equation}
\frac{\partial }{\partial t} q_t(x,y) = -\frac{\partial }{\partial
y} \bigl( \mu \coth (\mu y ) q_t (x,y)\bigr) + \frac{1}{2}
\frac{\partial ^2}{\partial y^2} q_t(x,y)\,.
\end{equation}
Hence, the generator of the process is given by
$$
\Ll f(y) = \mu \coth (\mu y) \frac{d}{dy} f(y) + \frac{1}{2}
\frac{d^2}{dy^2} f(y)
$$
and this implies the SDE (\ref{partigiani}).

\end{proof}

Let us consider now the Brownian motion with drift $-\mu$
Doob--conditioned to hit $h$ before $0$ and killed when it reaches
$h$. In order to precise its meaning when the Brownian motion starts
at the origin, given   $0<x<h$,  we define $\PP ^{\uparrow , \mu}
_{x,h}$ as the conditioned law on $C([0,\infty),\RR)$
\begin{equation}\label{pomme}
\PP ^{\mu, \uparrow } _{x,h}(\L) = \PP ^\mu  _x \bigl( \L\,|\,
T_h<T_0\bigr )\,,\qquad  \L \in \cup _{s\geq 0 } \Ff _s \,.
\end{equation}
Note that the above definition is well posed since by
(\ref{resistenza})  $ \PP ^\mu _x ( T_h<T_0 ) = W(x)/W(h)>0$.

\begin{lemma}\label{block}
Given $0<x<h$, let $Q^{\mu, \uparrow }_{x,h}$ and $R^{\mu,
\uparrow}_{x,h}$ be the law of the path $(X_t \,:\, 0\leq t \leq
T_h)$ killed when level $h$ is reached, where $X$ is chosen with law
$\PP^{\mu, \uparrow }_{x,h}$ and $\PP^{\mu, \uparrow}_x $
respectively. Then $Q^{\mu, \uparrow }_{x,h}= R^{\mu, \uparrow}_x$
and the weak limit $Q_{0,h}^{\mu, \uparrow}:=\lim_{x\downarrow 0}
Q_{x,h} ^{\mu,\uparrow}$ exists and equals $R_0^{\mu, \uparrow }$.
\end{lemma}
\begin{proof}
Given    $0<x_1,x_2, \dots, x_n <h$ and times $t_1< t_2< \cdots<
t_n$, we denote by $\Aa$ the event
$$
\Aa:=\left\{X_{t_1}\in d x_1, X_{t_2} \in d x_2, \dots , X_{t_n} \in
d x_n, t_n < T_h\right\}\,.
$$Then,
by  definition of $\PP^{\mu, \uparrow}_{x,h}$ and the Markov
property of the Brownian motion, for each $0<x<h$ we get that
$$
\PP_{x,h}^{\mu, \uparrow} (\Aa)=
 \PP _x ^\mu  (\Aa, t_n < T_0 )
\PP_{x_n} ^\mu (T_h<T_0)/ \PP^\mu _x (T_h<T_0) = \PP _x ^\mu  (\Aa,
t_n < T_0 )W(x_n)/W(x)\,.
$$
Due to (\ref{jonny}), the last expression in the r.h.s. equals the
probability $ \PP^{\mu,\uparrow }_x(\Aa) $, hence we can conclude
that $\PP_{x,h}^{\mu, \uparrow} (\Aa)=\PP_{x}^{\mu, \uparrow}
(\Aa)$. Hence $Q_{x,h}^{\mu,\uparrow }= R_{x,h} ^{\mu, \uparrow}$
for $0<x<h$. The last statement concerning $Q_{0,h}^{\mu, \uparrow}$
follows from the fact that the weak  limit $\lim_{x\downarrow 0}
\PP_x ^{\mu, \uparrow} $ exists and equals $\PP_0 ^{\mu, \uparrow
}$.
\end{proof}

Due to the first part of the above lemma, we can think of $\bigl(
\PP^{\mu, \uparrow} _{x,h}, 0\leq x\leq h \bigr)$ as the Brownian
motion with drift $-\mu$ Doob--conditioned to hit $h$ before $0$ and
killed when it hits $h$.

\medskip

We have now all the tools in order to describe the behavior of the
$h$--slopes not covering the origin, near to their extremes. In
order to simplify the  notation, in what follows we denote  by
$B^{(\mu)}$ the two--sided Brownian motion with drift $-\mu$,
starting at the origin. Moreover, given $r\in \RR$,  we define
\begin{align*}
& T_r ^{(h,+)} =\inf \left\{s>0\,:\, \left| B_{r+s}^{(\mu)} -B_r^{(\mu)}\right|=h \right\}\,,\\
& T_r ^{(h,-)} =\inf \left\{s>0\,:\, \left| B_{r-s}^{(\mu)} -B_r^{(\mu)}\right|=h \right\}\,.
\end{align*}



\begin{theo}\label{bonaparte}
Let $\mu\not=0$ and let  $m<m'$ be consecutive points  of
$h$--extrema for the drifted Brownian motion $B^{(\mu)}$, both non
negative or both non positive.

If $m$ is a pont of $h$--minimum and  $m'$ is a point of
$h$--maximum, then  the  processes
\begin{align}
 &  \left\{ B^{(\mu)} _{m+t } -B^{(\mu)} _{m} \,,\, 0\leq t\leq  T_m ^{(h,+)}    \right\}\,, \label{mela1}\\
 &   \left\{ B^{(\mu)} _{m'} -  B^{(\mu)} _{m'-t} \,,\, 0\leq t\leq      T_{m'} ^{(h,-)}     \right\} \,, \label{mela2}
 \end{align}
have the same law of the Brownian motion starting at the origin,
with drift $-\mu$, Doob--conditioned to reach $+\infty$ before $0$
and killed when it hits $h$. Moreover, they  have the same law of
the Brownian motion starting at the origin, with drift $-\mu$,
Doob--conditioned to reach $h$ before $0$ and killed when it hits
$h$. In particular, they  satisfy the SDE (\ref{partigiani}) up to
the  killing  time.


 If $m$ is a point of $h$--maximum  and  $m'$ is a point of
$h$--minimum, then the  processes
\begin{align}
&   \left\{ B^{(\mu)} _{m} - B^{(\mu)} _{m+t} \,,\, 0\leq t\leq     T_m ^{(h,+)}    \right\} \,, \label{pera1}\\
 &  \left\{ B^{(\mu)} _{m'-t } -B^{(\mu)} _{m'} \,,\, 0\leq t\leq   T_{m'} ^{(h,-)}  \right\}\,,
 \label{pera2}
\end{align}
have the same law of the Brownian motion starting at the origin,
with drift $\mu$, Doob--conditioned to reach $+\infty$ before $0$
and killed when it hits $h$. Moreover, they  have the same law of
the Brownian motion starting at the origin, with drift $\mu$,
Doob--conditioned to reach $h$ before $0$ and killed when it hits
$h$. In particular, they satisfy the SDE (\ref{partigiani}) with
$\mu$ replaced by $-\mu$,  up to the killing  time.
\end{theo}
\begin{proof}
The second part of the theorem follows from the first part by taking
the reflection w.r.t. the coordinate axis.

As follows from the proof of Lemma \ref{davide} in Section
\ref{pierpa}, the law of the process (\ref{mela1})  coincides with
the law of the excursion $f$ killed when it reaches $h$, where $f$
is chosen with probability measure
$$
n( \cdot | T_h<T_0 )\,= \, n \bigl( \cdot ,  T_h < T_0 \bigr) /
n(T_h< T_0 )\,.
$$
 Due to Proposition 15 in \cite{B}[Section VII.3],
there exists a positive constant $c$ such that
\begin{align}
&  n \bigl( \L , t< T_0 \bigr) \,=\,  c\, \EE _0 ^{\mu, \uparrow }
\bigl(
W(X_t)^{-1} , \L \bigr) \,, \qquad \forall \L \in \Ff _t\,, \label{crick}\\
& n ( T_h < T_0 ) = c /W(h)\label{crock}
\end{align}
(note that (\ref{crock}) corresponds to (\ref{kuka1}) with $c=2$).
Hence,  given  numbers $x_1, x_2, \dots, x_n$ in $(0,h)$ and
increasing times $0<t_1<t_2 < \cdots < t_n$, we have
\begin{equation}\label{cruck}
\begin{split}
& n \bigl( f_{t_1} \in d x_1, \, f_{t_2} \in
 d x_2,\, \dots,\, f_{t_n} \in d x_n \,, t_n <T_h < T_0 \bigr) =\\
 &
 n \bigl( f_{t_1} \in d x_1, \, f_{t_2} \in
 d x_2,\, \dots,\, f_{t_n} \in d x_n \,, t_n <T_h , t_n < T_0 \bigr)
 \PP ^\mu _{x_n} ( T_h < T_0)=\\
 & c\, \PP^{\mu, \uparrow }_0 \bigl(
f_{t_1} \in d x_1, \, f_{t_2} \in
 d x_2,\, \dots,\, f_{t_n} \in d x_n \,, t_n <T_h \bigr)
 /W(h)
\end{split}
\end{equation}
(the first identity follows from the Markov property (\ref{rino2}),
while the latter follows from (\ref{resistenza}) and (\ref{crick})).
From
 (\ref{crock}) and (\ref{cruck}) one derives that
 \begin{multline}
n \bigl( f_{t_1} \in d x_1, \, f_{t_2} \in
 d x_2,\, \dots,\, f_{t_n} \in d x_n \,, t_n <T_h | T_h< T_0
 \bigr)=\\
\PP^{\mu, \uparrow }_0 \bigl( f_{t_1} \in d x_1, \, f_{t_2} \in
 d x_2,\, \dots,\, f_{t_n} \in d x_n \,, t_n <T_h \bigr)\,.
 \end{multline}
This concludes the proof that the process (\ref{mela1}) has the same
law of the Brownian motion starting at the origin, with drift
$-\mu$, Doob--conditioned to reach $+\infty$ before $0$ and killed
when it hits $h$. All the other statements follow from this
property,
  Lemmata \ref{black} and \ref{block} and by reflection arguments.

\end{proof}

\section{Proof of Lemma \ref{kuka}}\label{primavera}
%


Knowing the Laplace transform of the hitting times of the non--drifted Brownian motion
 \cite{KS}[Chapter 2], the following lemma follows by applying  Girsanov formula
 (\ref{aperol}):
 \begin{lemma} \label{corvo}
Let $x<0<y$ and $\hat \a >0$, then
\begin{align}
& \EE _0 ^{\mu} \left( e^{-\a T_x }\II_{T_x<T_y } \right)=
 e^{-\mu x }\frac{ \sinh \left( y \sqrt{2 \hat \a } \right) }{\sinh \left( (y-x) \sqrt{ 2\hat \a}\right) },\label{papera1}
\\
&
 \EE _0 ^{\mu} \left( e^{-\a T_y }\II_{T_y<T_x } \right)=
 e^{-\mu y  }\frac{ \sinh \left( -x  \sqrt{2 \hat \a } \right) }{\sinh \left( (y-x)\sqrt{ 2\hat \a}\right)}
\label{papera2} . \end{align}
\end{lemma}
\begin{proof}
Set $Z_t =\exp \left\{ -\mu B_t-\mu^2 t/2\right\}$.
We claim that
\begin{multline*}
\EE _0 ^{\mu} \left( e^{-\a T_x }\II_{T_x<T_y } \right)=
\lim _{t\uparrow\infty} \EE _0 ^{\mu} \left( e^{-\a T_x }\II_{T_x<T_y } \II _{T_x <t } \right)=\\
\lim _{t\uparrow\infty} \EE_0   \left( e^{-\a T_x }\II_{T_x<T_y } \II _{T_x <t } Z_t \right)=
\lim _{t\uparrow\infty} \EE_0   \left( e^{-\a T_x }\II_{T_x<T_y } \II _{T_x <t }
\EE_0(Z_t \,|\,  \Ff_{t\wedge T_x}  )\right).
\end{multline*}
Indeed, the first identity follows from the Monotone  Convergence
Theorem, the second one from (\ref{aperol}), and the last one by
conditioning on  $\Ff_{t\wedge T_x}$ and observing that $e^{-\a T_x
} \II_{T_x<T_y } \II _{T_x <t }$ is $\Ff_{t\wedge T_x}$ measurable.

Since, under $\PP _0$,  $Z_t$ is a martingale  and $t\wedge T_x$ is
a bounded stopping time, the optional sampling theorem implies that
$ \EE\left(Z_t \,|\,  \Ff_{t\wedge T_x}  \right)= Z_{ t\wedge T_x}.
$ For $T_x<t$,  $ Z_{t\wedge T_x }$ equals $ \exp\left( -\mu x -\hat
\a T_x +\a T_x\right) $, hence
$$
\EE _0 ^\mu \left( e^{-\a T_x }\II_{T_x<T_y } \right)=\lim
_{t\uparrow\infty} e^{-\mu x} \EE_0   \left( e^{-\hat \a T_x
}\II_{T_x<T_y } \II _{T_x <t } \right) = e^{-\mu x} \EE_0 \left(
e^{-\hat \a T_x }\II_{T_x<T_y }  \right).
$$
Since $\hat\a>0$,
the above identity together with formula (8.27)  in \cite{KS}[Chapter 2] implies (\ref{papera1}).

The proof of (\ref{papera2}) can be obtained by similar arguments and by
 formula (8.28) in \cite{KS}[Chapter 2] or simply by observing that
 $$
 \EE _0 ^{\mu} \left( e^{-\a T_y}\II_{T_y<T_x } \right)=
 \EE _0 ^{-\mu } \left( e^{-\a T_{-y }}\II_{T_{-y}<T_{-x} } \right)
 $$
and then applying (\ref{papera1}).
\end{proof}

Due to the above lemma, given $0<y<h$ and $\hat \a>0$,
\begin{align}
& \EE _y ^{\mu} \left( e^{-\a T_0 }\II_{T_0<T_h } \right)=
 e^{\mu y }\frac{ \sinh \left( (h-y ) \sqrt{2 \hat \a } \right) }{\sinh \left( h \sqrt{ 2\hat \a}\right)
 },
\label{leone1}\\
&
 \EE _y ^{\mu} \left( e^{-\a T_h }\II_{T_
 h<T_0 } \right)=
 e^{\mu (y-h) }\frac{ \sinh \left( y  \sqrt{2 \hat \a } \right) }{\sinh \left( h \sqrt{ 2\hat \a}\right)}\label{leone2},\\
 & \PP _y^\mu \left (T_0 < T_h \right) = e^{\mu y} \frac{ \sinh\left(\mu(h-y)\right)}{\sinh(\mu h )}
 \label{leone3} ,\\
 & \PP _y ^\mu \left (T_h < T_0 \right) = e^{\mu (y-h)} \frac{ \sinh\left(\mu y\right)}{\sinh(\mu h)}.
 \label{leone4}
 \end{align}
 By taking the limit
   $h\rightarrow\infty$ in (\ref{leone1}) we get for all $y>0$ and $\hat \a >0$
 \begin{equation}\label{leone5}
 \EE _y ^{\mu} \left( e^{-\a T_0 }\II_{T_0<\infty} \right)=
 e^{\mu y -y\sqrt{2\hat\a }}=e^{\mu y -|\mu| y \sqrt{ 1 + \frac{2\a}{\mu^2 }}} .
\end{equation}
By considering the  Taylor expansion around $\a=0$ in  above
identity, one can compute the expectation  of
$T_0^k\,\II_{T_0<\infty}$. In particular, for all $y >0$ it holds
\begin{equation}\label{leone6} \EE_y ^\mu ( T_0 \II_{T_0<\infty} ) =
e^{y(\mu-|\mu|)}  y/|\mu|.
\end{equation}

\medskip

 We collect some   identities (obtained by straightforward
computations) which will be very useful below. First we observe that
given $a,b,w \in \RR$ and $t>0$ it holds
\begin{equation}
\int_a ^b \frac{2y }{\sqrt{2\p t^3 }}e^{ -\frac{(y+w t )^2 }{2t}}dy
=\frac{2}{\sqrt{2\p t}} e^{ -\frac{(a+w t )^2
}{2t}}-\frac{2}{\sqrt{2\p t}} e^{ -\frac{(b+w t )^2 }{2t}}-\frac{2w
}{\sqrt{2\p}} \int _{a/\sqrt{t}+w\sqrt{t}}  ^{b/\sqrt{t}+w \sqrt{t}}
e^{-\frac{z^2}{2}}dz. \label{pipolo}
\end{equation}
In particular,  fixed $a>0$ and $c,w  \in \RR$, it holds as
$t\downarrow 0$
 \begin{align} &  \int _0 ^a\frac{2y }{\sqrt{2\p t^3
}}e^{ -\frac{(y+w t )^2 }{2t}}dy =\frac{2}{\sqrt{2\p t} } - w + o(1) \label{kaka1}\\
&  \int _a ^\infty \frac{2y }{\sqrt{2\p t^3 }}e^{ -\frac{(y+w t )^2
}{2t}}dy =  o(1)\label{kaka2}\\
&  \int_0 ^a \frac{2y }{\sqrt{2\p t^3 }}e^{ -\frac{(y+w  t )^2
}{2t}-c y
 }dy=\frac{2}{\sqrt{2\p t} } -( w+c) + o(1)\; \text{ as } \; t \downarrow 0
 \,.\label{kaka3}
 \end{align}
Note that the last identity can be derived from (\ref{pipolo}) by
observing that
 $$
 \int_0 ^a \frac{2y }{\sqrt{2\p t^3 }}e^{ -\frac{(y+w  t )^2
}{2t}-c y
 }dy=e^{\frac{t}{2}(c^2 +2 c w )}e^{ -\frac{(y+(w+c)  t )^2
}{2t}}\,.
$$

As first application of the above observations  and (\ref{rino2}),
we prove the following result:
\begin{lemma}\label{perdonanza}Let  $\hat \a>0 $.  Then
\begin{equation}
\lim _{t\downarrow 0} \int _U \bigl| 1 - e ^{-\a H(f)} \bigr| \II _{ H(f) \leq t  }n (df) =0\,.\label{basilico1}
\end{equation}
\end{lemma}
\begin{proof}
Since for a suitable positive constant $c>0$ it holds $|1-e^y|\leq c
|y|$ if $|y|\leq 1$, the limit
 (\ref{basilico1}) is implied by
 \begin{equation*}
 \lim _{t\downarrow 0} \int _U   H(f) \II _{ H(f) \leq t }n (df) =0.
 \end{equation*}
 Since  excursions are continuous paths and $\lim _{t\downarrow 0} \II _{ H(f) \leq t } =0$
 pointwise, by the  Dominated Convergence Theorem in order   to prove  the above limit it is enough to show that
 \begin{equation}\label{basilico4}
  \int _U   H(f) \II _{ H(f)\leq 1}  n (df) <\infty.
 \end{equation}
 By the Monotone Convergence Theorem
 \begin{equation}\label{basilico5}
  \int _U   H(f)  \II _{H(f)\leq 1} n (df) = \lim  _{\e\downarrow 0} \int  _U H(f)\II_{ \e< H(f)\leq 1 }   n (df) ,
  \end{equation}
 and due to (\ref{rino2})
 \begin{multline}\label{basilico6}
  \int_U   H(f)\II_{ \e< H(f) \leq 1 }   n (df)=
 \int n_\e (dy ) \EE _y ^\mu ((T_0+\e) \II _{T_0\leq 1-\e}  )
 \leq \\
 \int n_\e (dy ) \EE _y ^\mu (T_0 \II _{T_0<\infty }  ) +\e \int n_\e(dy)
  .
 \end{multline}
 Due to (\ref{kaka1}) and (\ref{kaka2}), the last term $\e \int n_\e
 (dy) $ is negligible as $\e \downarrow 0$
 while, due to  (\ref{leone6}),
    \begin{equation}\label{brahms10}
  \int n_\e (dy ) \EE _y ^\mu (T_0 \II _{T_0<\infty }   )
 = \frac{ 1}{\e |\mu| } \int _0 ^\infty \frac{2 y ^2}{\sqrt{ 2 \p \e} } e ^{y(\mu-|\mu|) -(y+\mu \e)^2 / (2\e)  } dy.
 \end{equation}
In order to conclude it is enough to observe that
 \begin{multline*}
 \text{R.h.s. of (\ref{brahms10})} =
 \frac{ 1}{\e |\mu| } \int _0 ^\infty \frac{2 y ^2}{\sqrt{ 2 \p \e} } e ^{-\frac{(y+|\mu| \e)^2}{ 2\e}  }
 dy\leq\\\frac{ 1}{\e |\mu| }\int _0^\infty   \frac{2 (y+|\mu|\e) ^2}{\sqrt{ 2 \p \e} } e ^{-\frac{(y+|\mu| \e)^2 }{2\e}  }
 dy= \frac{1}{|\mu|}<\infty\,.
 \end{multline*}


     \end{proof}

\bigskip
Now we have  all the technical tools in order to prove Lemma \ref{kuka}.
\subsection{Proof of (\ref{kuka1})}

Consider the measurable subsets $ U_t ^{h,+}=\bigl\{ f\in U: \sup
_{s\geq t} f_s\geq h,\, H(f)>t \bigr\} $. Then $U_{t_2}^{h,+}\subset
U_{t_1}^{h,+}$ for $t_1\leq t_2 $ and, by the continuity of
excursions, $U^{h,+}=\cup _{t>0 } U_t ^{h,+}. $ This implies that $
n \bigl( U^{h,+ } \bigr)= \lim _{t\downarrow 0 }   n \bigl( U_t^{h,+
} \bigr)$ .
 Due to (\ref{rino2} )
 \begin{equation*}
 n \bigl( U_t^{h,+ } \bigr)= \int _0 ^\infty n_t (dy )
\PP _y ^\mu \left( {\sup}_{s\geq 0}
 B_{s\wedge T_0} \geq h \right)
 =
   \int _0 ^\infty n_t (dy ) \PP _y ^\mu \left(
 T_h<T_0\right) . \end{equation*}
 Hence (see also  (\ref{leone4}))  we get that
\begin{equation*}
n \bigl( U^{h,+ } \bigr)   = \lim _{t\downarrow 0 } \left( I_1(t) +I_2(t)\right),
  \end{equation*}
where
$$ I_1(t) =\int _h ^\infty \frac{2y}{\sqrt{2\p t^3 }} e^{-\frac{(y+\mu t )^2 }{2t}}dy
,\qquad   I_2(t)= \int _0 ^h \frac{2y}{\sqrt{2\p t^3 }}
e^{-\frac{(y+\mu t )^2 }{2t}} e^{\mu (y-h)}\frac{ \sinh (\mu y )
}{\sinh (\mu h ) }dy \,.$$
Due to (\ref{kaka2})
$\lim _{t\downarrow 0} I_1 (t) = 0$. In order to treat the term
$I_2$ we write
$$
I_2 (t) = \left( I_3 (t) - I_4 (t) \right) / ( 1- e^{2\mu h }),
$$
where
\begin{align*}
&I_3(t) = \int _0^h  \frac{2y}{\sqrt{2\p t^3 }} e^{-\frac{(y+\mu t )^2 }{2t}}dy ,\\
& I_4(t) =  \int _0^h  \frac{2y}{\sqrt{2\p t^3 }} e^{-\frac{(y+\mu t )^2 }{2t}+2\mu y }dy=
\int _0^h  \frac{2y}{\sqrt{2\p t^3 }} e^{-\frac{(y-\mu t )^2 }{2t}}dy.
\end{align*}
Due to (\ref{kaka1}), $ I_3 (t) =  \frac{2}{\sqrt{2\p t }} -\mu
+o(1)$ and $  I_4 (t) =  \frac{2}{\sqrt{2\p t }} +\mu +o(1)$.
 In particular,
$$n\bigl(U^{h,+}\bigr)=\lim _{t\downarrow 0 } I_2(t) = \frac{2\mu}{e^{2\mu h}-1 }= \frac{\mu e^{-\mu h }}{\sinh (\mu h )}  ,
$$
thus concluding the proof of (\ref{kuka1}).

\qed

%
%

\subsection{Proof of (\ref{kuka4})}
Consider the subsets $U_t\subset U$ defined as $ U_t = \bigl\{ f\in
U: H(   f ) =\infty,\, \sup _{s\geq t} f_s <h \bigr \}.
$
Then $ U^{h,-}\cap U_\infty \subset U_t $ and, in order to prove
(\ref{kuka4}), it is enough to show that $ \lim _{t\downarrow 0} n
(U_t ) = 0 $.  Due to (\ref{rino2}) we can write
$$
n ( U_t) =\int _0 ^\infty n_t (dy) \PP _y^\mu ( T_h=T_0=\infty).
$$
Note that if $\mu>0$ then $ \PP _y^\mu ( T_0=\infty)=0$ for all $y>0$, while
if   $\mu<0$ then $ \PP _y^\mu ( T_h=\infty)=0$ for all $y<h$. Hence
$$
n(U_t)=\int _h ^\infty n_t (dy) \PP _y^\mu ( T_h=T_0=\infty)\leq  \int _h ^\infty n_t (dy)  .
$$
By (\ref{kaka2}) the last member above goes to $0$ as $t\downarrow
0$.  This implies that  $ n (U_t ) =o (1) $, thus concluding the
proof of (\ref{kuka4}).

\qed

\subsection{Proof of (\ref{kuka2})}


We claim that
\begin{equation}\label{violinoooo}
  \int _{U^{h,-}} \left( 1-e^{-\a H(f)}\right) n(df) =
  \lim _{t\downarrow 0}
  \int _U  \left( 1-e^{-\a H(f)}\right)\II_{ \{\sup _{[t,\infty)} f < h \}} \II _{ H(f)>t } n (df ) .
  \end{equation}
  Indeed, the above identity  follows from the Dominated Convergence Theorem if we prove that
  \begin{equation}\label{ragione}
   \int _{U^{h,-}} \bigl|1-e^{-\a H(f)}\bigr| n(df)  <\infty.
  \end{equation}
  To this aim,
  we observe that, for all $t>0$,
    \begin{equation}\label{doma}
    \int _{U^{h,-}} \bigl|1-e^{-\a H(f)}\bigr| n(df) \leq   \int _U  \bigl| 1-e^{-\a H(f)}\bigr| \II_{ \{\sup _{[t,\infty)} f < h \}} n (df ) ,
  \end{equation}
  and by (\ref{basilico1})
  \begin{multline}\label{doma1}
  \limsup _{t\downarrow 0 } \left( \text{r.h.s. of } (\ref{doma})\right)=  \limsup _{t\downarrow 0}
  \int _U  \bigl|1-e^{-\a H(f)}\bigr| \II_{ \{\sup _{[t,\infty)} f < h \}} \II _{ H(f)>t }  n (df ) =\\
  \limsup  _{t\downarrow 0}
 \text{sgn} (\a)   \int _U  \bigl(1-e^{-\a H(f)}\bigr)\II_{ \{\sup _{[t,\infty)} f < h \}} \II _{ H(f)>t }  n (df )   .\end{multline}

We claim that
\begin{equation}\label{mosca}
 \lim _{t\downarrow 0}
    \int _U  \bigl(1-e^{-\a H(f)}\bigr)\II_{ \{\sup _{[t,\infty)} f < h \}} \II _{ H(f)>t }  n (df )  = \sqrt{2\hat \a}
    \coth  \left( h \sqrt{ 2 \hat \a }\right) -\mu \coth (\mu h)   .
        \end{equation}
    Note that (\ref{mosca}) implies  that the r.h.s. of (\ref{doma}) is bounded and this implies  (\ref{ragione}), which implies  (\ref{violinoooo}), which together with (\ref{mosca}) implies    (\ref{kuka2}).

    \bigskip

    Let us prove (\ref{mosca}). By (\ref{rino2})
 \begin{equation}\label{russia}
 \text{l.h.s. of } (\ref{mosca})=
  \lim _{t\downarrow 0} \int _0 ^\infty n_t (y ) \EE _y ^\mu \left( \II _{T_0 <T_h }
 \left( 1 - e ^{-\a T_0+\a t }\right) \right) dy  =\lim _{t\downarrow 0 }\left( J_1 (t) -e^{\a t} J_2 (t) \right),
 \end{equation}
 where
 \begin{align*}
 &
 J_1 (t) = \int _ 0 ^\infty n_t (y )  \PP _y ^\mu \left ( T_0 < T_h \right)dy ,\\
 &
 J_2 (t) =  \int _0 ^\infty n_t (y ) \EE _y ^\mu \left( \II _{T_0 <T_h }
  e ^{-\a T_0 } \right)dy= \int _0 ^h  n_t (y ) \EE _y ^\mu \left( \II _{T_0 <T_h }
  e ^{-\a T_0 } \right)dy   .
 \end{align*}
 Due to the identities derived at the beginning of  the proof of (\ref{kuka1}) we can
 write
 $$
  \lim _{t\downarrow 0 }
  \int _ 0 ^\infty n_t (y )  \PP _y ^\mu \left ( T_h < T_0 \right) = n\bigl (U^{h,+}\bigr)=  \frac{ \mu e^{-\mu h }}{\sinh (\mu h ) }   ,
  $$
while due to  (\ref{kaka1}) and (\ref{kaka2})
$$
\int _0 ^\infty n_t (y )dy =\frac{2}{\sqrt{2\p t }} -\mu + o(1).
$$
The above identities give
\begin{equation}\label{15marzo}
J_1 (t) = \frac{2}{\sqrt{2\p t }} -\mu  - \frac{ \mu e^{-\mu h }}{\sinh (\mu h ) } +  o(1)=\frac{2}{\sqrt{2\p t }} -\mu \coth (\mu h).
  \end{equation}
 Due to (\ref{leone1})
 \begin{multline*}
 J_2 (t) =
  \int _0 ^h    \frac{2y}{\sqrt{2\p t^3 }} e^{-\frac{(y+\mu t )^2 }{2t}+\mu y}
  \frac{ \sinh \left( (h-y ) \sqrt{2 \hat \a } \right) }{\sinh \left( h \sqrt{ 2\hat \a}\right) }
  dy   =\\
  \frac{ e^{h\sqrt{2\hat \a}}}{ 2\sinh \left( h \sqrt{ 2 \hat \a }\right) } \int _0 ^h
   \frac{2y}{\sqrt{2\p t^3 }} e^{-\frac{(y+\mu t )^2 }{2t}- y\left(  \sqrt{2\hat \a} -\mu\right)}dy \\-
    \frac{ e^{-h\sqrt{2\hat \a}}}{ 2\sinh \left( h \sqrt{ 2 \hat \a }\right) } \int _0 ^h
   \frac{2y}{\sqrt{2\p t^3 }} e^{-\frac{(y+\mu t )^2 }{2t}+ y\left( \sqrt{2\hat \a}+\mu
   \right)}dy  .   \end{multline*}
Hence,  due to (\ref{kaka3}),
\begin{multline}\label{15marzobis}
 J_2 (t) = \frac{ e^{h\sqrt{2\hat \a}}}{ 2\sinh \left( h \sqrt{ 2 \hat \a }\right) }
\left( \frac{2}{\sqrt{2\p t }} -\sqrt{ 2\hat \a} + o(1) \right) -   \frac{ e^{-h\sqrt{2\hat \a}}}{ 2\sinh \left( h \sqrt{ 2 \hat \a }\right) }
\left( \frac{2}{\sqrt{2\p t }}-\sqrt{ 2\hat \a} + o(1) \right)=\\
  \frac{2}{\sqrt{2\p t }}-\sqrt{2\hat \a}
    \coth  \left( h \sqrt{ 2 \hat \a }\right)  .
      \end{multline}
      Then (\ref{mosca}) follows from (\ref{russia}),
 (\ref{15marzo}) and (\ref{15marzobis}), thus concluding the proof of (\ref{kuka2}).
 \qed
\subsection{Proof of (\ref{kuka3})}
Due to the Monotone Convergence Theorem
\begin{equation}\label{sassone}
 \int _U e^{-\a T_h }\II_{T_h<T_0}  n(df) = \lim _{t\downarrow 0}
 \int _U e^{-\a T_h }\II_{t<T_h<T_0}    n(df) .
  \end{equation}
It is convenient to write the last integral as $ A(t)-B(t)$,
 where
 \begin{align*}
 &  A(t)=
 e^{-\a t}
\int _U e^{-\a T_h (\theta_t f)  }\II_{T_h(\theta_t f)<T_0(\theta_t f)} \II _{H(f)>t}  n(df)\, ,\\
& B(t)=e^{-\a t}
\int _U e^{-\a T_h (\theta_t f)  }\II_{T_h(\theta_t f)<T_0(\theta_t f)}\II _{H(f)>t}\II _{T_h\leq t }  n(df)\,.
\end{align*}
Hence
\begin{equation}\label{airberlin}
\int _U e^{-\a T_h }\II_{T_h<T_0}  n(df) = \lim _{t\downarrow 0}
A(t)-\lim _{t\downarrow 0} B(t)\,.
\end{equation}
By    (\ref{rino2})
\begin{equation}\label{corno0}
\lim _{t\downarrow 0} A(t)= \lim _{t\downarrow 0} e^{-\a t} \int _0 ^\infty n_t (dy)
 \EE _y ^\mu
 \left( e^{-\a T_h } \II _{T_h < T_0 }\right)=\lim _{t\downarrow 0} \left(K_1 (t)+K_2(t)\right),
 \end{equation}
where
\begin{align*}
 K_1(t) & =  \int _0 ^h n_t (dy)  \EE _y ^\mu \left( e^{-\a T_h } \II _{T_h < T_0 }\right),\\
 K_2(t) & =  \int _h ^\infty n_t (dy)  \EE _y ^\mu \left( e^{-\a T_h } \II _{T_h < T_0 }\right)=
 \int _h ^\infty n_t (dy)  \EE _y ^\mu \left( e^{-\a T_h } \II _{T_h<\infty} \right) \\
& =
  \int _h ^\infty n_t (dy)  \EE _{y-h} ^\mu \left( e^{-\a T_0 }  \II _{T_0<\infty}     \right)
 .
\end{align*}
Due to (\ref{leone2})
  \begin{multline}
K_1(t)=
 \int _0^h
  \frac{2y}{\sqrt{2\p t^3 }} e^{-\frac{(y+\mu t )^2 }{2t}+\mu (y-h ) }\frac{
 \sinh\left(y\sqrt{2\hat \a }\right)}{   \sinh\left(h \sqrt{2\hat \a }\right) }   dy=\\
 \frac{ e^{-\mu h }}{ 2  \sinh \left(h \sqrt{2\hat \a }\right) }
 \int_0 ^h
   \frac{2y}{\sqrt{2\p t^3 }} e^{-\frac{(y+\mu t )^2 }{2t}+\mu y  }\left( e^{y \sqrt{2\hat \a }}-e^{-y
    \sqrt{2\hat \a }}\right) dy\,.
     \end{multline}
     By applying   (\ref{kaka3}) to the r.h.s. we get that
   \begin{equation}\label{corno1}
    K_1 (t) = \frac{
e^{-\mu h } \sqrt{ 2 \hat \a }
}{ \sinh \left( h \sqrt{ 2\hat \a }\right) }+ o(1).
\end{equation}
  By (\ref{leone5})
  \begin{equation}
   K_2 (t) = \int_h ^\infty
   \frac{2y}{\sqrt{2\p t^3 }} e^{-\frac{(y+\mu t )^2 }{2t}+(\mu  - \sqrt{2\hat \a})(y-h) }dy=   e^{\a t+h (\sqrt{2\hat\a} -\mu)   }\int_h ^\infty
    \frac{2y}{\sqrt{2\p t^3 }} e^{-\frac{(y+\sqrt{2 \hat \a} t )^2 }{2t} }dy
   \end{equation}
  and due to (\ref{kaka2})
 $ \lim _{t\downarrow 0} K_2 (t) =0$.
 This limit together with (\ref{corno0}) and  (\ref{corno1}) implies
 that
  \begin{equation}\label{grande1}
 \lim _{t\downarrow 0 } A(t)=   \frac{
e^{-\mu h } \sqrt{ 2 \hat \a }
}{ \sinh \left( h \sqrt{ 2\hat \a }\right) }.
\end{equation}
We claim that $ \lim _{t\downarrow 0} B(t) =0$. Note that this
together with   (\ref{airberlin}),  (\ref{grande1})  and
(\ref{kuka1}) implies (\ref{kuka3}). Hence, in order to conclude we
only need to prove  our claim. To this aim we apply H\"older
inequality with exponents  $p,q >1$ with $1/p+1/q=1$ and
$\widehat{\a p }
>0$, deriving that
$$ B(t) \leq e^{-\a t} \left[\int _U e^{-\a p  T_h (\theta_t f)  }\II_{T_h(\theta_t f)<T_0(\theta_t f)}\II _{H(f)>t}  n(df)
\right]^{1/p} n ( T_h \leq t )^{1/q}\,.
$$
As $t \downarrow 0$ the first factor in the r.h.s. goes to $1$, the
second factor has finite limit  due to (\ref{grande1}) where $\a$
has to be replaced by $\a p $ (here we use that $\widehat{\a p}
>0$). Hence we only need to prove that $ \lim _{t\downarrow 0} n (
T_h \leq t )=0$. By the Monotone Convergence Theorem it is enough to
show that $n(T_h \leq 1 )<\infty$.  But $n(T_h \leq 1) \leq n
(U^{h,+})$ which is bounded by (\ref{kuka1}).

\qed

\medskip

\medskip

\noindent {\bf Acknowledgements}. The author thanks A. Bovier, F. Le
Gall and Z. Shi for useful discussions. Moreover, she kindly
acknowledges the Weierstrass Institute for Applied Analysis and
Stochastics in Berlin   for the kind hospitality  while part of this
work was being done.

\end{document}